\documentclass[11pt]{article}
\usepackage{amsfonts}
\usepackage{graphics}
\usepackage{indentfirst}
\usepackage{color}
\usepackage{cite}
\usepackage{latexsym}
\usepackage[paper=a4paper, left=1.6cm, right=1.6cm, top=1.8cm, bottom=1.6cm, headheight=5.5pt, footskip=0.8cm, footnotesep=0.8cm, centering, includefoot]{geometry}
\usepackage{amsmath}
\allowdisplaybreaks
\usepackage{amssymb}
\usepackage[colorlinks, linkcolor=red]{hyperref}
\hypersetup{urlcolor=red, citecolor=red}
\usepackage[dvips]{epsfig}
\usepackage{amscd}

\usepackage{amsthm}
\usepackage{mathrsfs}
\usepackage{verbatim}

\DeclareMathOperator{\divv}{div}
\DeclareMathOperator{\curl}{curl}
\DeclareMathOperator{\loc}{loc}
\allowdisplaybreaks

\begin{document}
\title{Global weak solutions with higher regularity to the two-dimensional isentropic compressible Navier--Stokes and magnetohydrodynamic equations with far-field vacuum and unbounded density
\thanks{
This research was partially supported by National Natural Science Foundation of China (No. 12371227) and Fundamental Research Funds for the Central Universities (No. SWU--KU24001).
}
}

\author{Shuai Wang,\ Xin Zhong {\thanks{E-mail addresses: swang238@163.com (S. Wang),
xzhong1014@amss.ac.cn (X. Zhong).}}\date{}\\
\footnotesize
School of Mathematics and Statistics, Southwest University, Chongqing 400715, P. R. China}

\maketitle
\newtheorem{theorem}{Theorem}[section]
\newtheorem{definition}{Definition}[section]
\newtheorem{lemma}{Lemma}[section]
\newtheorem{proposition}{Proposition}[section]
\newtheorem{corollary}{Corollary}[section]
\newtheorem{remark}{Remark}[section]
\renewcommand{\theequation}{\thesection.\arabic{equation}}
\catcode`@=11 \@addtoreset{equation}{section} \catcode`@=12
\maketitle{}

\begin{abstract}
We establish the global existence of a class of weak solutions to the isentropic compressible Navier--Stokes and magnetohydrodynamic (MHD) equations on the whole plane under a suitably small initial energy. The solutions constructed here admit far-field vacuum and unbounded densities. Moreover, they possess an intermediate regularity regime between the finite-energy weak solutions of Lions--Feireisl and the framework of Hoff. This particularly extends our previous half-plane case with Dirichlet boundary conditions (arXiv:2601.11852) to the whole-plane MHD coupling, and we also generalize the works of Hoff (Comm. Pure Appl. Math. 55 (2002), pp. 1365--1407) and Suen and Hoff (Arch. Ration. Mech. Anal. 205 (2012), pp. 27--58) by allowing vacuum states and unbounded density. Our analysis lies in a new perspective that exploits the spatial integrability of the density and the resulting integrability of the pressure, together with the specific structure of the MHD system.
\end{abstract}

\textit{Key words and phrases}. Compressible Navier--Stokes equations; compressible magnetohydrodynamics; global weak solutions; vacuum.

2020 \textit{Mathematics Subject Classification}. 76W05; 76N10; 35A01.


\tableofcontents

\section{Introduction}
\subsection{Background and motivation}

In many physically relevant flows, the fluid is sufficiently conducting (or partially ionized) so that magnetic induction cannot be treated as a passive field. In this regime, hydrodynamics becomes magnetohydrodynamics by incorporating the Lorentz force into the momentum balance and coupling it with Faraday's induction law, so that the velocity field and the magnetic field evolve in a genuinely intertwined manner.
To be precise, the compressible MHD system can be viewed as a coupling between the compressible Navier--Stokes equations for the fluid variables and a reduced Maxwell system in the MHD approximation: the momentum equation is coupled to the magnetic field through the Lorentz force
(equivalently written in divergence form as $\mathbf{B}\cdot\nabla\mathbf{B}-\frac12\nabla|\mathbf{B}|^2$ under $\divv\mathbf{B}=0$), while the magnetic field itself evolves according to Maxwell's equations with the induction equation as the principal dynamical law. In two space dimensions, this coupling leads to the following compressible MHD equations
\begin{align}\label{a1}
\begin{cases}
\rho_t+\divv(\rho\mathbf{u})=0,\\
(\rho\mathbf{u})_t+\divv(\rho\mathbf{u}\otimes\mathbf{u})+\nabla P=
\mu\Delta\mathbf{u}+(\mu+\lambda)\nabla\divv\mathbf{u}+\mathbf{B}\cdot\nabla\mathbf{B}-\frac12\nabla|\mathbf{B}|^2,\\
\mathbf{B}_t+\mathbf{u}\cdot\nabla\mathbf{B}-\mathbf{B}\cdot\nabla\mathbf{u}+\mathbf{B}\divv\mathbf{u}=\nu\Delta\mathbf{B},\\
\divv\mathbf{B}=0.
\end{cases}
\end{align}

We consider the Cauchy problem for system \eqref{a1} posed on $\mathbb{R}^2$ with the initial data
\begin{equation}\label{a2}
(\rho,\rho\mathbf{u},\mathbf{B})|_{t=0}=(\rho_0,\rho_0\mathbf{u}_0,\mathbf{B}_0)(\mathbf{x}),\ \ \ \mathbf{x}\in\mathbb{R}^2,
\end{equation}
which satisfy the far-field condition
\begin{equation}\label{a3}
(\rho_0,\rho_0\mathbf{u}_0,\mathbf{B}_0)(\mathbf{x})\rightarrow(0,\mathbf{0},\mathbf{0}),\ \ \ \text{as}\ |\mathbf{x}|\rightarrow\infty.
\end{equation}
Here $\rho$, $\mathbf{u}=(u^1,u^2)$, and $\mathbf{B}=(B^1,B^2)$ denote the density, velocity, and magnetic field, respectively,
while the pressure $P$ is given by the equation of state
\begin{equation*}
  P=P(\rho)=a\rho^\gamma, \ \ a>0, \ \gamma>1.
\end{equation*}
The constants $\mu$ and $\lambda$ represent shear viscosity and bulk viscosity of the fluid, respectively, satisfying the physical restrictions
\begin{equation*}
\mu>0,\ \ \ \mu+\lambda\geq0,
\end{equation*}
and $\nu>0$ is the resistivity coefficient.

The compressible MHD system \eqref{a1} serves as a fundamental macroscopic model in settings where both compressibility and magnetic stresses are significant, including astrophysical and geophysical plasmas and laboratory flows of conducting fluids (see, e.g., \cite{C25,D17,G16}). Some applications often involve large spatial domains with tenuous far-field states, where vacuum may occur and strong degeneration can develop.
Compared with the isentropic compressible Navier--Stokes equations (i.e., $\mathbf{B}\equiv\mathbf{0}$), the MHD system is genuinely two-way coupled and cannot be viewed as a routine perturbation. The momentum balance in \eqref{a1} is modified by the Lorentz force $\mathbf{B}\cdot\nabla\mathbf{B}-\frac12\nabla|\mathbf{B}|^2$,
highlighting the complementary roles of magnetic tension and magnetic pressure, while the magnetic field evolves through the induction equation, where transport and stretching by the flow interact directly with $\nabla\mathbf{u}$ and are accompanied by resistive diffusion under
$\divv\mathbf{B}=0$. In the presence of vacuum, the interplay between degeneracy and these MHD-specific couplings makes global bounds in weak regularity classes particularly delicate. Accordingly, the analysis must exploit the intrinsic structure of the magnetic subsystem. This motivates the study of global well-posedness in weak regularity classes pursued in the present work.

We briefly review several results on multidimensional compressible Navier--Stokes and MHD systems that are most relevant to our work. A classical theme concerning the global well-posedness of smooth solutions is small perturbations for initial data around an equilibrium state. For the three-dimensional compressible Navier--Stokes equations, Matsumura and Nishida \cite{MN80,MN83} first obtained global classical solutions when the initial data are sufficiently close to a constant equilibrium in $H^3$.
For the two-dimensional MHD Cauchy problem, Kawashima \cite{K84} proved the global existence of classical solutions for initial data being sufficiently close to a non-vacuum equilibrium. Later on, in the spirit of the Matsumura--Nishida theory \cite{MN80,MN83}, Li--Yu \cite{LY11} and Zhang--Zhao \cite{ZZ10} independently established the global existence and uniqueness of classical solutions to the compressible MHD system in $\mathbb{R}^3$ for small $H^3(\mathbb{R}^3)$ perturbations around an equilibrium with strictly positive density. These small-perturbation theories are most natural in the non-vacuum regime, since vacuum induces degeneracy so that standard perturbative estimates fail.

Beyond the non-vacuum regime, the possible presence of vacuum is of comparable physical relevance but mathematically more delicate, since degeneracy near vacuum region complicates regularity and stability analysis. A cornerstone of general solutions theory is due to P.-L. Lions \cite{PL98}, who proved the global existence of finite-energy weak solutions for the compressible Navier--Stokes system with vacuum. His approach combines the renormalized techniques with the effective viscous flux arguments, and has inspired significant extensions. Feireisl and collaborators \cite{F04,FNP01} subsequently reduced the condition on the adiabatic exponent to $\gamma > \frac{N}{2}$ ($N$ is the spatial dimension) with the aid of oscillation defect measures. At the same time, under spherical symmetry or axisymmetry, Jiang and Zhang \cite{JZ01,JZ03} obtained global weak solutions for all $\gamma>1$  by deriving refined estimates for the effective viscous flux that yield improved integrability of the density. For the compressible MHD system, Hu and Wang \cite{HW10} established the global existence and large-time behavior of finite-energy weak solutions in three-dimensional bounded domains with Dirichlet boundary conditions for $\gamma>\frac{3}{2}$, by exploiting the magnetic coupling and building on the weak convergence framework of Lions \cite{PL98} and Feireisl \cite{F04}. Despite these advances, uniqueness for such weak solutions remains a fundamental and largely open problem.

Parallel to the development of weak solutions theory, there is now a substantial global well-posedness theory for strong (or classical) solutions in the presence of vacuum, provided the initial total energy is small even when the initial data exhibit large oscillations. A significant progress in this direction is due to Huang, Li, and Xin \cite{HLX12}, who obtained a global well-posedness result (existence and uniqueness) for classical solutions of the
three-dimensional compressible Navier--Stokes equations with vacuum under a small-energy assumption; this approach was later adapted to the three-dimensional compressible MHD system in \cite{LXZ13}.
In two dimensions, the far-field analysis is more delicate, in part because solutions exhibit distinct logarithmic behavior at spatial infinity. To overcome this difficulty, the strategy of \cite{HLX12} was extended to the whole plane in \cite{LX19} via spatial-weighted estimates and a refined large-time decay analysis. Meanwhile,
the global existence and uniqueness of strong solutions to the compressible MHD equations in $\mathbb{R}^2$ were established in \cite{LSX16} for small initial energy by assuming additionally a weighted regularity condition on the initial magnetic field.
More recently, the small-energy requirement has been relaxed in three dimensions. The case of Navier--Stokes system was treated in \cite{HHPZ24}, where the adiabatic exponent $\gamma$ approaches $1$.
An analogous improvement for the three-dimensional MHD Cauchy problem was obtained in \cite{HHPZ17} under a suitable smallness condition of the form
$\big[(\gamma-1)^{1/9}+\nu^{-1/4}\big]E_0$ with $E_0$ being the initial total energy. Overall, vacuum effects and magnetic coupling underscore that MHD is not merely a black-box extension of Navier--Stokes theory.

Turning back to weak solutions, an essential issue beyond existence is well-posedness, in particular uniqueness and stability. Addressing these questions typically requires working with weak solutions that possess additional regularity. Alongside the finite-energy weak solutions \cite{PL98,FNP01,HW10} and the small-energy classical solutions \cite{HLX12,LX19,LY11,LXZ13,ZZ10,LSX16}, a third important class is provided by the intermediate weak solutions introduced by Hoff \cite{Hoff95,Hoff95*}. These solutions enjoy higher regularity than the finite-energy class, for instance they allow the definition of particle trajectories away from vacuum, while still accommodating low-regularity density profiles including possible discontinuities \cite{Hoff02}. This enhanced regularity can be exploited to obtain well-posedness properties.
Hoff \cite{Hoff06} established uniqueness and continuous dependence, and Hoff and Santos \cite{Hoff08} further studied the Lagrangian structure and propagation of singularities.
For the compressible MHD system, Suen and Hoff \cite{SH12} proved global existence of such intermediate weak solutions in $\mathbb{R}^3$ for initial data that are small in $L^2(\mathbb{R}^3)$ and have strictly positive, essentially bounded initial density. This was later extended in \cite{LYZ13} to allow interior vacuum. The resulting regularity was subsequently used in \cite{S20} to prove uniqueness and continuous dependence, strengthening the existence theory of \cite{SH12}. Nevertheless, the construction of Hoff-type intermediate weak solutions under a far-field vacuum condition remains open. This regime is physically relevant for flows into a rarefied background, and it poses substantial analytical difficulties due to degeneracy as the density tends to zero at infinity.

In the far-field vacuum regime, where $\rho(\mathbf{x})\rightarrow0$ as $|\mathbf{x}|\rightarrow\infty$, one is led to a weak theory that can describe flows into a tenuous or rarefied background and allows vacuum and low regularity. Existing results in this setting are mainly at the strong-solution level and rely essentially on decay-based and higher-order estimates, as in \cite{LX19}. In the MHD case, the analogous result in \cite{LSX16} additionally requires suitable decay of the magnetic field, together with an essential use of the resistive diffusion structure.
Since such decay mechanisms are not available at the weak level, we are led to seek an alternative framework. Motivated by our previous work \cite{WZ26}, we propose to capture the far-field behavior through an $L^p$ framework of
the density with $\gamma<p<\infty$, which enforces integrable control of the far-field tails and seems more compatible with vanishing density at infinity than an $L^\infty$ bound. The goal of the present paper is to investigate whether this $L^p$ density framework can bridge the gap in the presence of magnetic coupling, and to clarify the role played by the magnetic field in overcoming the difficulties induced by vacuum both in the interior and at spatial infinity.

\subsection{Main results}

Before stating our main results, we first formulate the notations and conventions used throughout this paper. We denote by $C$ a generic positive constant which may vary at different places, and write
$C(f)$ to emphasize its dependence on $f$. The symbol $\Box$ marks the end of a proof and $a\triangleq b$ means $a=b$ by definition. For $1\le p\le \infty$ and integer $k\ge 0$, we denote the standard Sobolev spaces:
\begin{align*}
L^p=L^p(\mathbb{R}^2),\ W^{k, p}=W^{k, p}(\mathbb{R}^2),\ H^k=W^{k, 2}, \
D^{1,p}=\{f\in L_{\loc}^1(\mathbb{R}^2):\|\nabla f\|_{L^p}<\infty\}, \ D^k=D^{k,2}.
\end{align*}
For simplicity, we write
\begin{align*}
B_R\triangleq \{\mathbf{x}\in\mathbb{R}^2: |\mathbf{x}|<R\},\ \
\int f \mathrm{d}\mathbf{x}=\int_{\mathbb{R}^2} f \mathrm{d}\mathbf{x}, \ \ f_i=\partial_if=\frac{\partial f}{\partial x_i}.
\end{align*}
For two $n\times n$ matrices $A=\{a_{ij}\}$ and $B=\{b_{ij}\}$, the symbol $A: B$ represents the trace of the matrix product $AB^\top$, i.e.,
\begin{equation*}
A:B\triangleq\operatorname{tr}(AB^\top)=\sum_{i,j=1}^na_{ij}b_{ij}.
\end{equation*}

The initial total energy is defined as
\begin{equation}\label{1.4}
C_0\triangleq\int_{\mathbb{R}^2}
\bigg(\frac{1}{2}\rho_0|\mathbf{u}_0|^2+\frac{1}{2}|\mathbf{B}_0|^2
+\frac{1}{\gamma-1}P(\rho_0)\bigg)\mathrm{d}\mathbf{x}.
\end{equation}
Moreover, we denote by
\begin{align}\label{1.5}
\begin{cases}
\dot{f}\triangleq f_t+\mathbf{u}\cdot\nabla f,\\
F\triangleq(2\mu+\lambda)\divv\mathbf{u}-P(\rho)-\frac12|\mathbf{B}|^2,\\
\curl\mathbf{u}\triangleq-\nabla^\bot\cdot\mathbf{u}=-\partial_2u^1+\partial_1u^2,
\end{cases}
\end{align}
which represent the material derivative of $f$, the effective viscous flux, and the vorticity, respectively.

We recall the definition of weak solutions to the problem \eqref{a1}--\eqref{a3} in the sense of \cite{Hoff95,Hoff95*,SH12}.
\begin{definition}\label{d1.1}
A triplet $(\rho, \mathbf{u}, \mathbf{B})$ is said to be a weak solution to the problem \eqref{a1}--\eqref{a3} provided that
\begin{equation*}
(\rho, \rho\mathbf{u}, \mathbf{B}) \in C([0,\infty);H^{-1}(\mathbb{R}^2)),\ \
(\nabla\mathbf{u}, \nabla\mathbf{B})\in L^2(\mathbb{R}^2\times(0,\infty))
\end{equation*}
with $\divv \mathbf{B}(\cdot,t)=0$ in $\mathcal{D}'(\mathbb{R}^2)$ for $t>0$ and  $(\rho,\mathbf{u},\mathbf{B})|_{t=0}=(\rho_0,\mathbf{u}_0,\mathbf{B}_0)$.
Moreover, for any $t_2\ge t_1\ge 0$ and any $C^1$ test function $\psi({\bf x},t)$ with uniformly bounded support in ${\bf x}$ for $t\in[t_1,t_2]$, the following identities hold\footnote{Throughout this paper, we will use the Einstein summation over repeated indices convention.}:
\begin{align*}
&\int_{\mathbb{R}^2}\rho(\mathbf{x},\cdot)\psi(\mathbf{x},\cdot)
\mathrm{d}\mathbf{x}\Big|_{t_1}^{t_2}=\int_{t_1}^{t_2}\int_{\mathbb{R}^2}(\rho\psi_t+
\rho\mathbf{u}\cdot\nabla\psi)\mathrm{d}\mathbf{x}\mathrm{d}t,\\
&\int_{\mathbb{R}^2}(\rho u^j)(\mathbf{x},\cdot)\psi
(\mathbf{x},\cdot)\mathrm{d}\mathbf{x}\Big|_{t_1}^{t_2}
+\int_{t_1}^{t_2}
\int_{\mathbb{R}^2}\big(\mu\nabla u^j\cdot\nabla\psi+(\mu+\lambda)\divv\mathbf{u}\psi_j\big)
\mathrm{d}\mathbf{x}\mathrm{d}t\notag\\
&\qquad=\int_{t_1}^{t_2}
\int_{\mathbb{R}^2}\Big(\rho u^j\psi_t+\rho u^j\mathbf{u}\cdot\nabla\psi+P\psi_j
+\frac{1}{2}|\mathbf{B}|^2\psi_j-B^j\mathbf{B}\cdot\nabla\psi\Big)
\mathrm{d}\mathbf{x}\mathrm{d}t,\\
&\int_{\mathbb{R}^2} B^j(\mathbf{x},\cdot)\psi(\mathbf{x},\cdot)
\mathrm{d}\mathbf{x}\Big|_{t_1}^{t_2}=\int_{t_1}^{t_2}\int_{\mathbb{R}^2}\big(B^j\psi_t+
B^j\mathbf{u}\cdot\nabla\psi-u^j\mathbf{B}\cdot\nabla\psi-\nu\nabla B^j\cdot\nabla\psi\big)\mathrm{d}\mathbf{x}\mathrm{d}t.
\end{align*}
\end{definition}

Given $\alpha\in(1,2)$, assume that the initial data $(\rho_0,\mathbf{u}_0, \mathbf{B}_0)$ satisfies
\begin{equation}\label{1.6}
  0\leq\rho_0\in L^{\theta}, \ \ \bar{x}^\alpha\rho_0\in L^1, \ \ (\mathbf{u}_0,\mathbf{B}_0)\in D^1,\ \
  (\rho_0|\mathbf{u}_0|^2,|\mathbf{B}_0|^2)\in L^1, \ \ \divv\mathbf{B}_0=0,
\end{equation}
where
\begin{equation}\label{1.7}
\theta\triangleq \frac{4\gamma(2\alpha+1)}{\alpha-1}\in(20\gamma,\infty),
\ \ \bar{x}\triangleq\big(e+|\mathbf{x}|^2\big)^{1/2}\log^2\big(e+|\mathbf{x}|^2\big).
\end{equation}
Without loss of generality, we normalize the initial density $\rho_0$ so that
\begin{equation}\label{1.8}
  \int_{\mathbb{R}^2}\rho_0\mathrm{d}\mathbf{x}=1,
\end{equation}
which implies that there is a positive constant $\eta_0$ satisfying
\begin{equation}\label{1.9}
  \int_{B_{\eta_0}}\rho_0\mathrm{d}\mathbf{x}\geq\frac{1}{2}\int_{\mathbb{R}^2}\rho_0\mathrm{d}\mathbf{x}=\frac{1}{2}.
\end{equation}
Moreover, we suppose that there exist constants $\hat{\rho}\geq1$ and $M\geq1$ such that
\begin{align}\label{1.10}
\|\rho_0\|_{L^{\theta}}\leq \hat{\rho}, \ \
\|\bar{x}^\alpha\rho_0\|_{L^{1}}+\|\nabla\mathbf{u}_0\|_{L^2}+\|\nabla\mathbf{B}_0\|_{L^2}\leq M.
\end{align}

Now we state our main results on the global existence of weak solutions.

\begin{theorem}\label{t1.1}
Let the assumptions \eqref{1.6}, \eqref{1.8}, and \eqref{1.10} be satisfied. Then there is a positive constant $\varepsilon$ depending only on $\alpha, \hat{\rho}, M, a, \gamma, \mu, \lambda, \nu$, and $\eta_0$ such that if
\begin{equation}\label{1.11}
  C_0\leq\varepsilon,
\end{equation}
the problem \eqref{a1}--\eqref{a3} admits a global weak solution $(\rho,\mathbf{u},\mathbf{B})$ in the sense of Definition $\ref{d1.1}$ satisfying, for any $0<T<\infty$,
\begin{equation}\label{1.12}
\begin{cases}
0\leq\rho\in L^\infty(0,T; L^{\theta}(\mathbb{R}^2))\cap C([0,T];L^q(\mathbb{R}^2)), \ \ \text{for any} \ q\in[1,\theta),\\
\bar{x}^\alpha\rho\in L^\infty(0,T;L^{1}(\mathbb{R}^2)),\ \
(\sqrt{\rho}\mathbf{u},\mathbf{B})\in C([0,T];L^2(\mathbb{R}^2)),\ \
(\nabla\mathbf{u},\nabla\mathbf{B})\in L^2(\mathbb{R}^2\times(0,T)),\\
\big(\sigma^{\frac{1}{2}}\nabla\mathbf{u},\ \sigma^{\frac{1}{2}}\nabla\mathbf{B}, \ \sigma^{\frac{3}{2}}\sqrt{\rho}\dot{\mathbf{u}}\big)
\in L^\infty(0,T;L^{2}(\mathbb{R}^2)),\
\big(\sigma^{\frac{1}{2}}\sqrt{\rho}\dot{\mathbf{u}},\
\sigma^{\frac{1}{2}}\Delta\mathbf{B},\
\sigma^{\frac{3}{2}}\nabla\dot{\mathbf{u}}\big)\in L^2(\mathbb{R}^2\times(0,T)),
\end{cases}
\end{equation}
with $\sigma=\sigma(t)\triangleq\min\{1,t\}$, and
\begin{equation}\label{1.13}
  \inf_{0\leq t\leq T}\int_{B_{\eta_1(1+t)\log^{\alpha}(e+t)}}\rho(\mathbf{x},t)\mathrm{d}\mathbf{x}\geq\frac14
\end{equation}
for some positive constant $\eta_1$ depending only on $\alpha, \hat{\rho}, M, a, \gamma,\eta_0$, and $\|\sqrt{\rho_0}\mathbf{u}_0\|_{L^2}$.
\end{theorem}

If $\mathbf{B}\equiv\mathbf{B}_0\equiv\mathbf{0}$, Theorem \ref{t1.1} directly yields the following result on the global existence of weak solutions for the isentropic compressible Navier--Stokes system:
\begin{align}\label{b1}
\begin{cases}
\rho_t+\divv(\rho\mathbf{u})=0,\\
(\rho\mathbf{u})_t+\divv(\rho\mathbf{u}\otimes\mathbf{u})+\nabla P=
\mu\Delta\mathbf{u}+(\mu+\lambda)\nabla\divv\mathbf{u},\\
(\rho,\rho\mathbf{u})|_{t=0}=(\rho_0,\rho_0\mathbf{u}_0)(\mathbf{x}),
\end{cases}
\end{align}
with the far-field behavior
\begin{equation}\label{b2}
(\rho_0,\rho_0\mathbf{u}_0)(\mathbf{x})\rightarrow(0,\mathbf{0}),\ \ \ \text{as}\ |\mathbf{x}|\rightarrow\infty.
\end{equation}

\begin{theorem}\label{t1.2}
In addition to \eqref{1.8}, assume that
\begin{align*}
  0\leq\rho_0\in L^{\theta}, \ \ (\bar{x}^\alpha\rho_0, \rho_0|\mathbf{u}_0|^2)\in L^1, \ \ \mathbf{u}_0\in D^1, \ \
\|\rho_0\|_{L^{\theta}}\leq \bar{\rho}, \ \
\|\bar{x}^\alpha\rho_0\|_{L^{1}}+\|\nabla\mathbf{u}_0\|_{L^2}\leq \widetilde{M}
\end{align*}
for some constants $\bar{\rho}\geq1$ and $\widetilde{M}\geq1$.
Then there is a positive constant $\tilde{\varepsilon}$ depending only on $\alpha, \bar{\rho}, \widetilde{M}, a, \gamma, \mu$, $\lambda$, and $\eta_0$ such that if
\begin{equation*}
  C_0\leq\tilde{\varepsilon},
\end{equation*}
the Cauchy problem \eqref{b1}--\eqref{b2} admits a global weak solution $(\rho,\mathbf{u})$ satisfying, for any $0<T<\infty$,
\begin{equation*}
\begin{cases}
0\leq\rho\in L^\infty(0,T; L^{\theta}(\mathbb{R}^2))\cap C([0,T];L^q(\mathbb{R}^2)), \ \ \text{for any} \ q\in[1,\theta),\\
\bar{x}^\alpha\rho\in L^\infty(0,T;L^{1}(\mathbb{R}^2)),\ \
\sqrt{\rho}\mathbf{u}\in C([0,T];L^2(\mathbb{R}^2)),\ \
\nabla\mathbf{u}\in L^2(\mathbb{R}^2\times(0,T)),\\
\big(\sigma^{\frac{1}{2}}\nabla\mathbf{u},\ \sigma^{\frac{3}{2}}\sqrt{\rho}\dot{\mathbf{u}}\big)
\in L^\infty(0,T;L^{2}(\mathbb{R}^2)),\
\big(\sigma^{\frac{1}{2}}\sqrt{\rho}\dot{\mathbf{u}},\
\sigma^{\frac{3}{2}}\nabla\dot{\mathbf{u}}\big)\in L^2(\mathbb{R}^2\times(0,T)),
\end{cases}
\end{equation*}
and
\begin{equation*}
\inf_{0\leq t\leq T}\int_{B_{\hat{\eta}(1+t)\log^{\alpha}(e+t)}}\rho(\mathbf{x},t)\mathrm{d}\mathbf{x}\geq\frac14
\end{equation*}
for some positive constant $\hat{\eta}$ depending only on $\alpha, \bar{\rho}, \widetilde{M}, a, \gamma,\eta_0$, and $\|\sqrt{\rho_0}\mathbf{u}_0\|_{L^2}$.
\end{theorem}

Several remarks are in order.
\begin{remark}
It should be noted that the weak solutions constructed in Theorems \ref{t1.1} and \ref{t1.2} fall into an intermediate regularity class, lying between the finite-energy weak solutions in the sense of Lions--Feireisl and the intermediate weak solutions considered in \cite{Hoff95,Hoff95*,Hoff02,SH12,S20}.
In contrast, the presence of far-field vacuum appears to be beyond the scope of the approaches developed in \cite{Hoff02,LYZ13,SH12,S20}, which rely on uniform $L^\infty$ bounds of the density.
Our framework, based on an $L^\theta$ setting, extends their theories to allow for vacuum states both in the interior and at infinity as well as permitting unbounded densities.
\end{remark}

\begin{remark}
We point out that our strategies differ substantially from those of \cite{LX19,LSX16}, where the authors rely heavily on delicate large-time decay estimates to overcome difficulties caused by far-field vacuum:
\cite{LX19} treats the pressure and the velocity gradient, while \cite{LSX16} further requires decay control of $\nabla\mathbf{B}$. In contrast, within the weak solutions framework, any attempt to obtain decay control at spatial infinity is intrinsically delicate and generally unavailable. In view of this, we need to reevaluate our approach and instead make systematic use of the $L^q$ integrability ($1 < q=\theta/\gamma < \infty$) for the pressure $P(\rho)$ after controlling the magnetic field to mitigate the temporal growth of some quantities and capture the far-field behavior. Such mechanism also yields a more flexible treatment of the magnetic field, allowing us to dispense with the additional assumption in \cite{LSX16} that $\|\bar{x}^\alpha|\mathbf{B}_0|^2\|_{L^1}\leq M$.
\end{remark}

\begin{remark}
The index $\theta$ in \eqref{1.7} is a technical parameter arising from the weighted inequality \eqref{2.2} and the large-time analysis in \eqref{3.64}, which reflects a structural feature of our framework based on the modified spatial integrability of the density; see Subsection \ref{sec1.3} for further discussion. In particular, this integrability, transmitted through the pressure, feeds into the velocity estimates and then propagates to the magnetic field through the MHD coupling.
At the same time, the pressure interacts with the magnetic pressure, while magnetic tension and transport effects influence the field dynamics, so that the two subsystems affect each other in an essentially nontrivial way.
\end{remark}

\begin{remark}
In comparison with the whole plane case, the present framework may be expected to be more effective in periodic or bounded domains, where the Poincar\'e inequality and compact embeddings between $L^p$ spaces provide additional analytical advantages.
In fact, as shown in \cite{BB23}, global weak solutions on the torus $\mathbb{T}^N$ ($N=2,3$) have been constructed for compressible flows with an anisotropic viscous stress tensor. It should be emphasized that the structural properties of the magnetic field equation, and particularly the presence of magnetic diffusion, play crucial roles in our analysis. However, in the absence of magnetic diffusion (i.e., $\nu=0$), the problem becomes substantially more delicate: even for the incompressible system, local well-posedness relies on highly technical arguments; see, for instance, \cite{FMRR14,FMRR17}. Addressing these issues in greater generality would require new ideas and are left for future investigation.
\end{remark}

\subsection{Strategy of the proof}\label{sec1.3}

We now outline the main ideas and difficulties of the proof. Our argument is based on constructing global smooth approximate solutions and performing a careful limiting procedure. It proceeds in two steps. First, we invoke a local existence result for initial data with strictly positive density (Lemma \ref{l2.1}). Second, we let the lower bound of the initial density tend to zero (see Section \ref{sec4}). The key challenge is to bridge these steps by establishing \textit{a priori} estimates that are uniform with respect to this lower bound and yield the required regularity.

It should be emphasized that the key techniques developed in \cite{LX19,LSX16} cannot be adopted directly in the present setting. In their strong solution framework, the treatment of far-field vacuum essentially relies on time-decay rates and higher-order estimates for the pressure and gradients of the velocity and magnetic fields.
By contrast, the weak solution framework lacks a decay-sensitive compactness mechanism at spatial infinity so that such decay estimates are generally unavailable.
Consequently, the far-field vacuum problem calls for new observations and ideas.

From the mass equation $\eqref{a1}_1$, we find that, for any $1<\theta<\infty$,
\begin{equation}\label{1.18}
\frac{\mathrm{d}}{\mathrm{d}t}\int\rho^\theta\mathrm{d}\mathbf{x}=-(\theta-1)\int\rho^\theta\divv\mathbf{u}\mathrm{d}\mathbf{x}
=-\frac{\theta-1}{2\mu+\lambda}\int\rho^\theta
\Big(P+F+\frac12|\mathbf{B}|^2\Big)\mathrm{d}\mathbf{x},
\end{equation}
which suggests that the combination
\begin{equation*}
  F+\frac12|\mathbf{B}|^2=(2\mu+\lambda)\divv\mathbf{u}-P(\rho)
\end{equation*}
is fundamental in producing dissipation for the density at the $L^\theta$ level. Motivated by \eqref{1.18}, we work within an $L^\theta$-based integrability framework for the density in order to capture far-field behavior beyond the classical $L^\infty$ setting. One of the key ingredients is the weighted inequalities in Lemma \ref{l2.4} established in our previous work \cite[Lemma 2.4]{WZ26}, which require the nonnegative function $\varrho\in L^1\cap L^\zeta$. Accordingly, we assume $\rho\in L^1\cap L^\theta$ with $\theta\in[\zeta,\infty)$. Although the magnetic equation does not involve the density explicitly, the magnetic field is expected to remain tractable in the new framework provided that its intrinsic structure and resistive diffusion are fully exploited. However, without an $L^\infty$ bound of the density, the nonlinear terms involving $\rho$ or $P(\rho)$ become substantially harder to control, which calls for refined time-weighted estimates and a careful use of the available integrability of the density and the pressure.

The first and foremost difficulty lies in the fact that the density-weighted structure destroys many classical $L^p$-estimates. More precisely, the inequality \eqref{2.2} and Lemma \ref{l3.5} indicate that the upper bound of $\|\rho\dot{\mathbf{u}}\|_{L^p}$ ($2<p<\infty$) may grow polynomially in time. To obtain global-in-time control of the effective viscous flux $F$, we rely crucially on the velocity decomposition
$\mathbf{u}=\mathbf{w}_1+\mathbf{w}_2+\mathbf{w}_3$ given in Lemma \ref{l2.6}:
\begin{equation*}
\|F\|_{L^p}\lesssim \|\nabla\mathbf{u}\|_{L^p}+\cdots\lesssim \|\nabla\mathbf{w}_1\|_{L^p}+\cdots,
\end{equation*}
together with the Sobolev embedding in two dimensions:
\begin{equation*}
\|F\|_{L^6}\lesssim\|\nabla F\|_{L^\frac32}, \ \
 \|\nabla\mathbf{w}_1\|_{L^p}\lesssim
 \|\nabla^2\mathbf{w}_1\|_{L^{\frac{2p}{2+p}}}\lesssim \|\rho\dot{\mathbf{u}}\|_{L^{\frac{2p}{2+p}}}
 \lesssim \|\sqrt{\rho}\|_{L^p}\|\sqrt{\rho}\dot{\mathbf{u}}\|_{L^2}.
\end{equation*}
Nevertheless, these less tractable estimates inevitably lead to higher powers of $\|\sqrt{\rho}\dot{\mathbf{u}}\|_{L^2}$, which in turn require a more delicate time-weighted energy analysis with $\sigma=\sigma(t)\triangleq\min\{1,t\}$.
In particular, for integral terms such as
\begin{equation*}
\int_0^T\int\sigma P|\nabla\mathbf{u}|^2\mathrm{d}\mathbf{x}\mathrm{d}t,
\end{equation*}
we split the time interval into an initial layer and a long-time regime, and carefully isolate the singular behavior of the velocity gradient near the initial time using $L^6$-based bounds (see \eqref{3.23} and \eqref{3.24}).
This procedure allows us to reduce the analysis to controlling
\begin{equation*}
\int_{0}^{T}\int\sigma^3P^4\mathrm{d}\mathbf{x}\mathrm{d}t,
\end{equation*}
which can be bounded by the estimates in \eqref{3.45} once the effective viscous flux and the  magnetic contributions have been treated. Notably, a key mechanism in the nonlinear analysis is to exploit the coupling
$\mathbf{u}\cdot\nabla\mathbf{B}-\mathbf{B}\cdot\nabla\mathbf{u}$ in $\eqref{a1}_3$, which allows us to avoid a direct reliance on the $L^2$-norm
of the velocity alone. Achieving this relies essentially on the presence of magnetic diffusion term $\nu\Delta\mathbf{B}$ (see \eqref{3.33}--\eqref{3.36}).

Moreover, establishing any decay control at spatial infinity within the weak solution framework is particularly delicate. This difficulty stems from the limited strength of the available \textit{a priori} bounds, such as the crude control of $\|\nabla\mathbf{u}\|_{L^p}$ in \eqref{2.4}, together with the lack of higher-order estimates that would yield decay. In fact, arguing as in Lemma \ref{l2.4} and assuming $\varrho\in L^1\cap L^{\zeta}$, we obtain the growth bound of the density-weighted material derivative (see \eqref{2.2} and \eqref{3.61}):
\begin{equation*}
\|\nabla F\|_{L^4}\lesssim \|\rho\dot{\mathbf{u}}\|_{L^4}+\||\mathbf{B}||\nabla\mathbf{B}|\|_{L^4}\lesssim (1+t)^4\cdots.
\end{equation*}
Given the current tools at our disposal, the only usable information is the $L^\theta$ bound on the density. To capture the far-field behavior, we exploit the spatial integrability of the pressure and the magnetic field to
reduce the time growth (see \eqref{ef1}--\eqref{ef5}). Moreover, the interpolation inequality in Lemma \ref{l2.2} yields
\begin{equation*}
\|F\|_{L^\infty} \lesssim
\|\nabla F\|_{L^4}^{\frac{3}{16}}
\|F\|_{L^{\frac{52}{3}}}^{\frac{13}{16}}
\lesssim
\|\nabla F\|_{L^4}^{\frac{3}{16}}
\|\nabla F\|_{L^{\frac{52}{29}}}^{\frac{13}{16}}
 \lesssim
(1+t)^\frac34
\Big(\|\sqrt{\rho}\|_{L^{\frac{52}{3}}}^{\frac{13}{16}}
\|\sqrt{\rho}\dot{\mathbf{u}}\|_{L^2}^{\frac{13}{16}}
+\||\mathbf{B}||\nabla\mathbf{B}|\|_{L^{\frac{52}{29}}}^{\frac{13}{16}}\Big)\cdots.
\end{equation*}
The improved growth rate is now compatible with the time-weighted estimates (see \eqref{3.64}). Accordingly, the refined initial density integrability index $\theta$ is chosen so that
\begin{equation*}
\theta\geq\max\big\{\zeta=4(2\alpha+1)/(\alpha-1), \ 6\gamma, \ 26/3\big\}.
\end{equation*}
In the present work, we take $\theta=\zeta\gamma\in(20\gamma,\infty)$ for $\alpha\in(1,2)$.

Having resolved the above difficulties, we further apply Zlotnik's inequality (see Lemma \ref{lzlo}) to obtain a time-uniform upper bound for $\|\rho\|_{L^\theta}$. A key step is the global control of $\|F\|_{L^\infty}$ and $\|\mathbf{B}\|_{L^\infty}$ (see \eqref{3.62} and \eqref{3.64}), where the density-weighted estimate \eqref{2.2} and the spatial-weighted bound in Lemma \ref{l3.5} play central roles. With these \textit{a priori} estimates at hand, we deduce the desired regularity of the density, which is not covered by the Lions--Feireisl class of weak solutions. Finally, combining a standard compactness argument with the continuity method, we construct a global weak solution in the sense of Definition \ref{d1.1} satisfying \eqref{1.12}.

The rest of the paper is organized as follows. In the next section, we recall several preliminary results and basic inequalities that will be used throughout the paper. Section \ref{sec3} is devoted to deriving the necessary \textit{a priori} estimates. Finally, Section \ref{sec4} contains the proof of Theorem \ref{t1.1}.

\section{Preliminaries}\label{sec2}

In this section we collect some facts and elementary inequalities that will be used later.

\subsection{Auxiliary results and inequalities}

In this subsection we review some known facts and inequalities. First, similar to the proofs in \cite{MN80,LH15}, we have the following result concerning the local existence of strong solutions to the problem \eqref{a1}--\eqref{a3}.
\begin{lemma}\label{l2.1}
Assume that
\begin{equation*}
 \bar{x}^\alpha\rho_0\in L^1,\ \rho_0\in H^{2},\
(\mathbf{u}_0,\mathbf{B}_0)\in D^2\cap D^1, \ (\sqrt{\rho_0}\mathbf{u}_0, \mathbf{B}_0)\in L^2, \ \divv\mathbf{B}_0=0, \
 \inf\limits_{\mathbf{x}\in B_{R}}\rho_0(\mathbf{x})>0
\end{equation*}
for some large ball $B_{R}$, then there exists a small time $T$ such that the problem \eqref{a1}--\eqref{a3} admits a unique strong solution $(\rho,\mathbf{u},\mathbf{B})$ satisfying
\begin{equation*}
\bar{x}^\alpha\rho \in L^\infty(0,T; L^{1}),\ \
\rho\in C([0,T]; H^{2}),\
(\mathbf{u},\mathbf{B})\in C([0,T]; D^{2}\cap D^1), \ \inf_{B_{R}\times[0,T]}\rho(\mathbf{x},t)
\geq \frac12 \inf\limits_{\mathbf{x}\in B_{R}}\rho_0(\mathbf{x})
>0.
\end{equation*}
\end{lemma}

The following well-known Gagliardo--Nirenberg inequality (see \cite{NI1959}) will be used later.
\begin{lemma}\label{l2.2}
For $p\in [2, \infty)$, $q\in(1, \infty)$, and $r\in (2, \infty)$, there exists some generic constant $C>0$ which may depend on $p$, $q$, and $r$ such that, for $f\in H^1$ and $g\in L^q\cap D^{1,r}$,
\begin{gather*}
\|f\|_{L^p}\leq C\|f\|_{L^2}^\frac{2}{p}\|\nabla f\|_{L^2}^{1-\frac{2}{p}},\ \ \
\|g\|_{L^\infty}\leq C\|g\|_{L^q}^\frac{q(r-2)}{2r+q(r-2)}\|\nabla g\|_{L^r}^\frac{2r}{2r+q(r-2)}.
\end{gather*}
\end{lemma}

The following weighted $L^p$-estimates for functions in $D^1(\mathbb{R}^2)$ are stated in
\cite[Theorem B.1]{PL96}.
\begin{lemma}\label{l2.3}
For $m\in[2,\infty)$ and $\beta\in(1+m/2,\infty)$, there exists a positive constant $C$ such that for all $\mathbf{v}\in D^{1}(\mathbb{R}^2)$,
\begin{equation*}
\left(\int_{\mathbb{R}^2}\frac{|\mathbf{v}|^m}{e+|\mathbf{x}|^2}(\log(e+|\mathbf{x}|^2))^{-\beta}\mathrm{d}\mathbf{x}\right)^{1/m}
\leq C\|\mathbf{v}\|_{L^2(B_1)}+C\|\nabla\mathbf{v}\|_{L^2(\mathbb{R}^2)}.
\end{equation*}
\end{lemma}

As a direct consequence of Lemma \ref{l2.3} and Poincar\'e's inequality, we obtain the following estimates for $\|\mathbf{v}\|_{L^2(B_{\eta_*})}$ and $\|\varrho\mathbf{v}\|_{L^4(\mathbb{R}^2)}$ with a nonnegative function $\varrho \in L^1(\mathbb{R}^2)\cap L^\zeta(\mathbb{R}^2)$, whose proof can be adapted from our previous work \cite[Lemma 2.4]{WZ26} on the half-plane case.
\begin{lemma}\label{l2.4}
Let $\zeta=4(2\alpha+1)/(\alpha-1)$ with $\alpha\in(1,\infty)$ and $\bar{x}$ be as in \eqref{1.7}. Assume that the nonnegative function $\varrho\in L^1(\mathbb{R}^2)\cap L^\zeta(\mathbb{R}^2)$ satisfies
\begin{equation*}
M_2\leq \int_{B_{\eta_*}}\varrho\mathrm{d}\mathbf{x}\leq
\|\varrho\|_{L^1(\mathbb{R}^2)}=M_1,\ \ \|\varrho\|_{L^\zeta(\mathbb{R}^2)}\leq M_3,\ \
  \bar{x}^\alpha\varrho\in L^1(\mathbb{R}^2)
\end{equation*}
for positive constants $M_1, M_2, M_3$, and $\eta_*\geq1$. Then there exists a constant $C>0$ depending only on $M_1$, $M_2$, and $M_3$ such that, for $\mathbf{v}\in D^{1}(\mathbb{R}^2)$ verifying $\sqrt{\varrho}\mathbf{v}\in L^2(\mathbb{R}^2)$,
\begin{align}\label{2.1}
  \|\mathbf{v}\|_{L^2(B_{\eta_*})}^2&\leq C\eta_*^4
  \big(\|\sqrt{\varrho}\mathbf{v}\|_{L^2(B_{\eta_*})}^2+
  \|\nabla\mathbf{v}\|_{L^2(B_{\eta_*})}^2),
   \\
  \|\varrho\mathbf{v}\|_{L^4(\mathbb{R}^2)}^4
   &\leq
  C\eta_*^8(1+\|\bar{x}^\alpha\varrho\|_{L^{1}(\mathbb{R}^2)})
 \big(\|\sqrt{\varrho}\mathbf{v}\|_{L^2(\mathbb{R}^2)}^4+\|\nabla\mathbf{v}\|_{L^2(\mathbb{R}^2)}^4 \big).\label{2.2}
\end{align}
\end{lemma}

Finally, we introduce a variant of Zlotnik's inequality (see \cite{Zlo}) to derive the uniform (in time) upper bound of $\|\rho\|_{L^\theta}$.

\begin{lemma}\label{lzlo}
Let the function $y$ satisfy\footnote{Although \cite[Lemma 1.3(b)]{Zlo} states the result for $y'(t)=g(y) + b'(t)$, its proof actually works for the inequality as well.}
\begin{equation*}
y'(t)\leq g(y) + b'(t) \ \ \text{on} \ [0, T], \quad y(0) = y^{0},
\end{equation*}
with $g \in C(\mathbb{R})$ and $y, b \in W^{1,1}(0, T)$. If $g(\infty)=-\infty$ and
\begin{equation*}
b(t_{2})-b(t_{1}) \leq N_{0} + N_{1}(t_{2} - t_{1})
\end{equation*}
for all $0 \leq t_{1} < t_{2} \leq T$ with some $N_{0} \geq 0$ and $N_{1} \geq 0$, then
\begin{equation*}
y(t) \leq \max \{y^{0}, \xi^*\} + N_{0} < \infty \ \ \text{on} \ [0, T],
\end{equation*}
where $\xi^*$ is a constant such that
\begin{equation*}
g(\xi) \leq -N_{1} \ \ \text{for} \ \xi \geq \xi^*.
\end{equation*}
\end{lemma}

\subsection{$L^p$-estimates for $\nabla F$ and $\nabla\mathbf{u}$}

In this subsection, we derive several $L^p$-estimates for later use in the {\it a priori} analysis.
\begin{lemma}\label{l2.6}
Let $(\rho,\mathbf{u},\mathbf{B})$ be a smooth solution to the problem \eqref{a1}--\eqref{a3}. Then, for any $1< q<\infty$ and $2< p<\infty$, there exists a generic positive constant $C$ depending only on $q$, $p$, $\mu$, and $\lambda$ such that
\begin{gather}\label{2.3}
\|\nabla F\|_{L^q}\leq
C\|\rho\dot{\mathbf{u}}\|_{L^q}+C\||\mathbf{B}||\nabla\mathbf{B}|\|_{L^q},\\
\|\nabla\mathbf{u}\|_{L^p}\leq C\|\sqrt{\rho}\|_{L^p}\|\sqrt{\rho}\dot{\mathbf{u}}\|_{L^2}+C\|P\|_{L^p}+C\|\mathbf{B}\|_{L^{2p}}^2.\label{2.4}
\end{gather}
\end{lemma}
\begin{proof}
Using the vector identity $\Delta \mathbf{u}=\nabla\divv \mathbf{u}-\nabla^\bot\curl \mathbf{u}$, we get from $\eqref{a1}_2$ that
\begin{equation}\label{2.5}
  \Delta F=\divv(\rho\dot{\mathbf{u}}-\mathbf{B}\cdot\nabla\mathbf{B}).
\end{equation}
The standard $L^p$-estimate for the elliptic system \eqref{2.5} directly yields \eqref{2.3}.
Let
\begin{align}\label{2.6}
\begin{cases}
\mu\Delta\mathbf{w}_1+(\mu+\lambda)\nabla\divv\mathbf{w}_1=\rho\dot{\mathbf{u}}, \\
\mu\Delta\mathbf{w}_2+(\mu+\lambda)\nabla\divv\mathbf{w}_2=\nabla \big(P+\frac12|\mathbf{B}|^2\big),\\
\mu\Delta\mathbf{w}_3+(\mu+\lambda)\nabla\divv\mathbf{w}_3=-\divv(\mathbf{B}\otimes\mathbf{B}),\\
\divv \mathbf{B}=0.
\end{cases}
\end{align}
Accordingly, we set $\mathbf{u}=\mathbf{w}_1+\mathbf{w}_2+\mathbf{w}_3$. It thus follows from the elliptic theory for the Lam\'e system that
\begin{align}\label{2.7}
 \|\nabla\mathbf{u}\|_{L^p}&\leq \|\nabla\mathbf{w}_1\|_{L^p}+\|\nabla\mathbf{w}_2\|_{L^p}
 +\|\nabla\mathbf{w}_3\|_{L^p}\notag\\
 &\leq C(p)\|\nabla^2\mathbf{w}_1\|_{L^{\frac{2p}{2+p}}}+\|\nabla\mathbf{w}_2\|_{L^p}
 +\|\nabla\mathbf{w}_3\|_{L^p}\notag\\
 &\leq C\|\rho\dot{\mathbf{u}}\|_{L^{\frac{2p}{2+p}}}+C\|P\|_{L^p}+C\||\mathbf{B}|^2\|_{L^p}\\
 &\leq C\|\sqrt{\rho}\|_{L^p}\|\sqrt{\rho}\dot{\mathbf{u}}\|_{L^2}+C\|P\|_{L^p}
 +C\||\mathbf{B}|^2\|_{L^p}.\tag*{\qedhere}
\end{align}
\end{proof}

\section{\textit{A priori} estimates}\label{sec3}

In this section we establish necessary {\it a priori} bounds for the strong solution to \eqref{a1}--\eqref{a3}, whose existence is guaranteed by Lemma \ref{l2.1}.
These bounds are independent of the lower bound of $\rho$, the initial regularity, and the time of existence. More precisely, fix $T>0$ and let $(\rho,\mathbf{u},\mathbf{B})$ be the corresponding strong solution on $\mathbb{R}^2\times(0,T]$.

First, after multiplying \eqref{a1}$_1$ by a cut-off function and use a standard limit procedure, one finds that
\begin{align}\label{3.1}
\|\rho(t)\|_{L^1}=\|\rho_0\|_{L^1},\ \ 0\leq t\leq T.
\end{align}
Next, set $\sigma=\sigma(t)\triangleq\min\{1,t\}$, and we define auxiliary functionals as
\begin{align}\label{3.2}
  &A_1(T)\triangleq \sup_{t\in[0,T]}\big[\sigma\big(\|\nabla \mathbf{u}\|_{L^2}^2+\|\nabla \mathbf{B}\|_{L^2}^2\big)\big]
  +\int_0^T\sigma\big(\|\sqrt{\rho}\dot{\mathbf{u}}\|_{L^2}^2
  +\|\Delta\mathbf{B}\|_{L^2}^2\big)\mathrm{d}t,\\
   &A_2(T)\triangleq \sup_{t\in[0,T]}\big[\sigma^3\big(\|\sqrt{\rho}\dot{\mathbf{u}}\|_{L^2}^2
   +\||\mathbf{B}||\nabla \mathbf{B}|\|_{L^2}^2\big)\big]
   +\int_0^T\sigma^3\big(\|\nabla \dot{\mathbf{u}}\|_{L^2}^2
   +\||\mathbf{B}||\Delta \mathbf{B}|\|_{L^2}^2\big)\mathrm{d}t,\label{3.3}\\
   &A_3(T)\triangleq \sup_{t\in[0,T]}\big(\|\nabla \mathbf{u}\|_{L^2}^2+\|\nabla \mathbf{B}\|_{L^2}^2\big)
  +\int_0^T\big(\|\sqrt{\rho}\dot{\mathbf{u}}\|_{L^2}^2
  +\|\Delta\mathbf{B}\|_{L^2}^2\big)\mathrm{d}t.
   \label{3.4}
\end{align}
Then we derive the following key {\it a priori} estimates on $(\rho, \mathbf{u},\mathbf{B})$.

\begin{proposition}\label{p3.1}
Under the conditions of Theorem $\ref{t1.1}$, there exist positive constants $K\geq 2M^2$ and $\varepsilon$ depending only on $\alpha, \hat{\rho}, M, a, \gamma, \mu, \lambda, \nu$, and $\eta_0$ such that if $(\rho, \mathbf{u},\mathbf{B})$ is a strong solution to the problem \eqref{a1}--\eqref{a3} satisfying
\begin{align}\label{3.5}
\sup_{\mathbb{R}^2\times[0,T]}\|\rho\|_{L^{\theta}}\le2\hat{\rho},\ \
A_1(T)+A_2(T)\leq 2C_0^\frac{\alpha+2}{7\alpha+5}, \ \
A_3(\sigma(T))\leq 3K,
\end{align}
then the following improved bounds hold
\begin{align}\label{3.6}
\sup_{\mathbb{R}^2\times[0,T]}\|\rho\|_{L^{\theta}}\le\frac{7}{4}\hat{\rho},\ \
A_1(T)+A_2(T)\leq C_0^\frac{\alpha+2}{7\alpha+5},\ \
A_3(\sigma(T))\leq 2K,
\end{align}
provided $C_0\leq \varepsilon$.
\end{proposition}

Before proving Proposition \ref{p3.1}, we establish some necessary \textit{a priori} estimates, see Lemmas \ref{l3.1}--\ref{l3.6} below. Let us begin with the basic energy estimate for
$(\rho, \mathbf{u},\mathbf{B})$.

\begin{lemma}\label{l3.1}
It holds that
\begin{align}\label{3.7}
\sup_{0\le t\le T}\int\bigg(\frac{1}{2}\rho |\mathbf{u}|^2+\frac{1}{2} |\mathbf{B}|^2+\frac{P}{\gamma-1}\bigg)\mathrm{d}\mathbf{x}+\int_0^T\left[\mu\|\nabla \mathbf{u}\|_{L^2}^2+(\mu+\lambda)\|\divv \mathbf{u}\|_{L^2}^2+\nu\|\nabla \mathbf{B}\|_{L^2}^2\right]\mathrm{d}t
\leq C_0.
\end{align}
\end{lemma}
\begin{proof}
From the mass equation $\eqref{a1}_1$, we have
\begin{equation}\label{3.8}
\frac{P_t}{\gamma-1}+\frac{\divv(P\mathbf{u})}{\gamma-1}+P\divv \mathbf{u}=0.
\end{equation}
Multiplying $\eqref{a1}_2$ by $\mathbf{u}$ and $\eqref{a1}_3$ by $\mathbf{B}$, respectively, adding the summation to \eqref{3.8}, and integrating the resultant over $\mathbb{R}^2$, we arrive at
\begin{align*}
\frac{\mathrm{d}}{\mathrm{d}t}\int\bigg(\frac{1}{2}\rho|\mathbf{u}|^2+\frac{1}{2} |\mathbf{B}|^2+\frac{P}{\gamma-1}\bigg)\mathrm{d}\mathbf{x}
+\int\left[\mu|\nabla\mathbf{u}|^2+(\mu+\lambda)(\divv\mathbf{u})^2
+\nu|\nabla\mathbf{B}|^2\right]\mathrm{d}\mathbf{x}
=0.
\end{align*}
Integrating the above equality with respect to $t$ over $(0,T)$ yields \eqref{3.7}.
\end{proof}

The next result concerns the preliminary bounds of $A_1(T)$ and $A_2(T)$.
\begin{lemma}\label{l3.2}
Let $(\rho,\mathbf{u},\mathbf{B})$ be a smooth solution of \eqref{a1}--\eqref{a3} on $\mathbb{R}^2\times(0,T]$ satisfying $\|\rho\|_{L^{\theta}}\le2\hat{\rho}$. Then there exist positive constants $C$ and $\varepsilon_1$ depending only on $\alpha, \hat{\rho}, M, a, \gamma, \mu, \lambda$, and $\nu$ such that
\begin{gather}\label{3.10}
A_1(T)\leq  CC_0^\frac{2(\alpha+2)}{7\alpha+5}+C\int_0^T \sigma\|\nabla\mathbf{u}\|_{L^3}^3\mathrm{d}t
 +C\int_{0}^T\sigma^3\|P\|_{L^4}^4\mathrm{d}t,
\\
A_2(T)\leq CA_1(T)+CA_1^2(T)+C
 \Big(A_1(T)+A_1^\frac32(T)\Big)A_2^\frac12(T)
 +C\int_0^T\sigma^3\|P\|_{L^4}^4\mathrm{d}t, \label{3.11}
\end{gather}
provided $C_0\leq \varepsilon_1$.
\end{lemma}
\begin{proof}
Using H\"older's inequality and Lemma \ref{l3.1}, we obtain the following bound of $\|P\|_{L^q}$ with $q\in[1,20]$:
\begin{equation}\label{3.12}
 \|P\|_{L^q}=a\|\rho\|_{L^{q\gamma}}^\gamma\leq  a\|\rho\|_{L^\gamma}^\frac{4\gamma(2\alpha+1)-q\gamma(\alpha-1)}{q(7\alpha+5)}
 \|\rho\|_{L^{\theta}}^\frac{4\gamma(q-1)(2\alpha+1)}{q(7\alpha+5)}
 \leq C(a,\gamma,\hat{\rho})C_0^\frac{4(2\alpha+1)-q(\alpha-1)}{q(7\alpha+5)}, \ \ 0\leq t\leq T.
\end{equation}
The momentum equation $\eqref{a1}_2$ can be rewritten as
\begin{equation}\label{3.13}
\rho\dot{\mathbf{u}}+\nabla P=\mu\Delta\mathbf{u}+(\mu+\lambda)\nabla\divv \mathbf{u}
+\mathbf{B}\cdot\nabla\mathbf{B}-\frac12\nabla|\mathbf{B}|^2.
\end{equation}

\textbf{Estimate for $A_1(T)$}.
Multiplying \eqref{3.13} by $\sigma\dot{\mathbf{u}}$ and integrating the resultant over $\mathbb{R}^2$, one gets that
\begin{align}\label{3.14}
\int\sigma\rho |\dot{\mathbf{u}}|^2\mathrm{d}\mathbf{x}&=-\int\sigma\dot{\mathbf{u}}\cdot\nabla P\mathrm{d}\mathbf{x}+
\int\sigma\big[\mu\Delta\mathbf{u}\cdot\dot{\mathbf{u}}
+(\mu+\lambda)\dot{\mathbf{u}}\cdot\nabla\divv \mathbf{u}\big]\mathrm{d}\mathbf{x}\notag\\
&\quad
-\frac12\int\sigma\dot{\mathbf{u}}\cdot\nabla|\mathbf{B}|^2
\mathrm{d}\mathbf{x}
+\int\sigma\mathbf{B}\cdot\nabla\mathbf{B}\cdot\dot{\mathbf{u}}\mathrm{d}\mathbf{x}
\triangleq\sum_{i=1}^{4}\mathcal{I}_i.
\end{align}
For the first two terms on the right-hand side of \eqref{3.14}, it follows from $\eqref{a1}_1$ that
\begin{align}\label{3.15}
\mathcal{I}_1
&=\int \sigma P\divv\mathbf{u}_t\mathrm{d}\mathbf{x}
-\int\sigma\mathbf{u}\cdot\nabla\mathbf{u}\cdot\nabla P\mathrm{d}\mathbf{x}
\notag\\
&=\frac{\mathrm{d}}{\mathrm{d}t}\int\sigma P\divv\mathbf{u} \mathrm{d}\mathbf{x}
-\int\sigma\divv\mathbf{u}P'(\rho)\rho_t\mathrm{d}\mathbf{x}
-\int\sigma' P\divv\mathbf{u} \mathrm{d}\mathbf{x}
-\int\sigma\mathbf{u}\cdot\nabla\mathbf{u}\cdot\nabla P\mathrm{d}\mathbf{x}\notag\\
&=\frac{\mathrm{d}}{\mathrm{d}t}\int \sigma P\divv\mathbf{u}\mathrm{d}\mathbf{x}+
\int\sigma(\divv\mathbf{u})^2P'(\rho)\rho\mathrm{d}\mathbf{x}+
\int\sigma\mathbf{u}\cdot\nabla P\divv\mathbf{u}\mathrm{d}\mathbf{x}
-\int\sigma' P\divv\mathbf{u} \mathrm{d}\mathbf{x}\notag\\
&\quad
-\int\sigma\mathbf{u}\cdot\nabla\mathbf{u}\cdot\nabla P\mathrm{d}\mathbf{x}\notag\\
&=\frac{\mathrm{d}}{\mathrm{d}t}\int\sigma P\divv\mathbf{u}\mathrm{d}\mathbf{x}+
\int\sigma(\divv\mathbf{u})^2(P'(\rho)\rho-P)\mathrm{d}\mathbf{x}
+\int\sigma P\partial_iu^j\partial_ju^i\mathrm{d}\mathbf{x}
-\int\sigma' P\divv\mathbf{u} \mathrm{d}\mathbf{x}\notag\\
&\leq\frac{\mathrm{d}}{\mathrm{d}t}\int\sigma P\divv\mathbf{u}\mathrm{d}\mathbf{x}+
C\int\sigma P|\nabla\mathbf{u}|^2\mathrm{d}\mathbf{x}
+C\int\sigma' P|\nabla\mathbf{u}|\mathrm{d}\mathbf{x},\\
\mathcal{I}_2
&=
-\frac{\mu}{2}\frac{\mathrm{d}}{\mathrm{d}t}\big(\sigma\|\nabla\mathbf{u}\|^2_{L^2}\big)
+\frac{\mu}{2}\sigma'\|\nabla\mathbf{u}\|^2_{L^2}
-\mu\int\sigma\partial_iu^j\partial_i(u^k\partial_ku^j)\mathrm{d}\mathbf{x}
-\frac{\mu+\lambda}{2}\frac{\mathrm{d}}{\mathrm{d}t}\big(\sigma\|\divv\mathbf{u}\|^2_{L^2}\big)\notag\\
&\quad+\frac{\mu+\lambda}{2}\sigma'\|\divv\mathbf{u}\|^2_{L^2}
-(\mu+\lambda)\int\sigma\divv \mathbf{u}\divv (\mathbf{u}\cdot\nabla\mathbf{u})\mathrm{d}\mathbf{x}\notag\\
&\leq -\frac{1}{2}\frac{\mathrm{d}}{\mathrm{d}t}\big[\mu\sigma\|\nabla\mathbf{u}\|^2_{L^2}+(\mu+\lambda)\sigma\|\divv\mathbf{u}\|^2_{L^2}\big]
+C\|\nabla\mathbf{u}\|^2_{L^2}+C\sigma\|\nabla\mathbf{u}\|^3_{L^3},\label{3.16}
\end{align}
owing to $0\leq \sigma,\sigma'\leq 1$.
Using $\eqref{a1}_3$ and \eqref{3.7}, we get that
\begin{align}\label{3.17}
 \mathcal{I}_3&=\frac12\int\sigma|\mathbf{B}|^2\divv\mathbf{u}_t\mathrm{d}\mathbf{x}
 +\frac12\int\sigma|\mathbf{B}|^2\divv(\mathbf{u}\cdot\nabla\mathbf{u})\mathrm{d}\mathbf{x}\notag\\
 &=\frac12\frac{\mathrm{d}}{\mathrm{d}t}\int\sigma|\mathbf{B}|^2\divv\mathbf{u}\mathrm{d}\mathbf{x}
-\frac12\int\sigma'|\mathbf{B}|^2\divv\mathbf{u}\mathrm{d}\mathbf{x}
+\frac12\int\sigma|\mathbf{B}|^2\partial_iu^j\partial_ju^i\mathrm{d}\mathbf{x}
 -\frac12\int\sigma|\mathbf{B}|^2(\divv\mathbf{u})^2\mathrm{d}\mathbf{x}\notag\\
 &\quad+\int\sigma\big(-\mathbf{B}\cdot\nabla\mathbf{u}+\mathbf{B}\divv\mathbf{u}
 -\nu\Delta\mathbf{B}\big)\cdot\mathbf{B}\divv\mathbf{u}\mathrm{d}\mathbf{x}\notag\\
 &\leq \frac12\frac{\mathrm{d}}{\mathrm{d}t}\int\sigma|\mathbf{B}|^2\divv\mathbf{u}\mathrm{d}\mathbf{x}
 +C\|\nabla\mathbf{u}\|^2_{L^2}+C\|\nabla\mathbf{B}\|^2_{L^2}+C\sigma\|\nabla\mathbf{u}\|_{L^3}^2
 \|\mathbf{B}\|_{L^2}^\frac43\|\Delta\mathbf{B}\|_{L^2}^\frac23+\frac{\nu}{16}\sigma\|\Delta\mathbf{B}\|_{L^2}^2
 \notag\\
 &\leq \frac12\frac{\mathrm{d}}{\mathrm{d}t}\int\sigma|\mathbf{B}|^2\divv\mathbf{u}\mathrm{d}\mathbf{x}
 +C\|\nabla\mathbf{u}\|^2_{L^2}+C\|\nabla\mathbf{B}\|^2_{L^2}+C\sigma\|\nabla\mathbf{u}\|_{L^3}^3
 +\frac{\nu}{8}\sigma\|\Delta\mathbf{B}\|_{L^2}^2.
\end{align}
A similar argument based on integration by parts yields that
\begin{align}\label{3.18}
 \mathcal{I}_4&=-\frac{\mathrm{d}}{\mathrm{d}t}\int\sigma\mathbf{B}\cdot\nabla\mathbf{u}\cdot\mathbf{B}\mathrm{d}\mathbf{x}
 +\int\sigma'\mathbf{B}\cdot\nabla\mathbf{u}\cdot\mathbf{B}\mathrm{d}\mathbf{x}
 +2\int\sigma(\mathbf{B}\otimes\mathbf{B}_t):\nabla\mathbf{u}\mathrm{d}\mathbf{x}
 +\int\sigma\mathbf{B}\cdot\nabla\mathbf{B}\cdot(\mathbf{u}\cdot\nabla)\mathbf{u}\mathrm{d}\mathbf{x}\notag\\
&\leq -\frac{\mathrm{d}}{\mathrm{d}t}\int\sigma\mathbf{B}\cdot\nabla\mathbf{u}\cdot\mathbf{B}\mathrm{d}\mathbf{x}
 +C\|\nabla\mathbf{u}\|^2_{L^2}+C\|\nabla\mathbf{B}\|^2_{L^2}+C\sigma\|\nabla\mathbf{u}\|_{L^3}^3
 +\frac{\nu}{8}\sigma\|\Delta\mathbf{B}\|_{L^2}^2.
\end{align}
In addition, multiplying $\eqref{a1}_3$ by $\sigma\Delta\mathbf{B}$ and integration by parts, we have
\begin{align}\label{3.19}
&\frac12\frac{\mathrm{d}}{\mathrm{d}t}\int\sigma|\nabla \mathbf{B}|^2\mathrm{d}\mathbf{x}+
\int\nu\sigma|\Delta \mathbf{B}|^2\mathrm{d}\mathbf{x}\notag\\
&\leq C\int\sigma'|\nabla \mathbf{B}|^2\mathrm{d}\mathbf{x}
+C\int\sigma|\nabla \mathbf{u}||\nabla \mathbf{B}|^2\mathrm{d}\mathbf{x}
+C\int\sigma|\nabla \mathbf{u}||\mathbf{B}||\Delta\mathbf{B}|\mathrm{d}\mathbf{x}\notag\\
&\leq C\|\nabla\mathbf{B}\|^2_{L^2}+C\sigma\|\nabla\mathbf{u}\|_{L^3}
 \Big(\|\nabla\mathbf{B}\|_{L^2}^\frac43\|\Delta\mathbf{B}\|_{L^2}^\frac23+
 \|\mathbf{B}\|_{L^2}^\frac23\|\Delta\mathbf{B}\|_{L^2}^\frac43\Big)
 \notag\\
&\leq C\|\nabla\mathbf{B}\|^2_{L^2}+C\sigma\|\nabla\mathbf{B}\|_{L^2}^4
+C\sigma\|\nabla\mathbf{u}\|_{L^3}^3+\frac{\nu}{4}\sigma\|\Delta\mathbf{B}\|_{L^2}^2.
\end{align}
Thus, substituting \eqref{3.15}--\eqref{3.18} into \eqref{3.14} and adding \eqref{3.19} multiplied by $(2\widetilde{C}+1)$, we obtain that
\begin{align}\label{3.20}
&\frac{\mathrm{d}}{\mathrm{d}t}\big(\sigma \mathcal{E}_1(t)+\widetilde{C}\sigma\|\nabla\mathbf{B}\|^2_{L^2}\big)+\frac12\int\sigma\big(\rho |\dot{\mathbf{u}}|^2+\nu|\Delta \mathbf{B}|^2\big)\mathrm{d}\mathbf{x}\notag\\
&\leq C\big(\|\nabla\mathbf{u}\|^2_{L^2}+\|\nabla\mathbf{B}\|^2_{L^2}\big)
+C\sigma\big(\|\nabla\mathbf{u}\|^3_{L^3}+\|\nabla\mathbf{B}\|_{L^2}^4\big)
+C\int\sigma' P|\nabla\mathbf{u}|\mathrm{d}\mathbf{x}
+C\int\sigma P|\nabla\mathbf{u}|^2\mathrm{d}\mathbf{x},
\end{align}
where
\begin{align}\label{3.21}
  \mathcal{E}_1(t)&\triangleq \frac{\mu}{2}\|\nabla\mathbf{u}\|^2_{L^2}+\frac{\mu+\lambda}{2}\|\divv\mathbf{u}\|^2_{L^2}
  +\frac{1}{2}\|\nabla\mathbf{B}\|^2_{L^2}
  -\int P\divv\mathbf{u}\mathrm{d}\mathbf{x}
  -\frac12\int|\mathbf{B}|^2\divv\mathbf{u}\mathrm{d}\mathbf{x}+\int\mathbf{B}\cdot\nabla\mathbf{u}\cdot\mathbf{B}\mathrm{d}\mathbf{x}\notag\\
  &\geq\frac{\mu}{2}\|\nabla\mathbf{u}\|^2_{L^2}+\frac{\mu+\lambda}{2}\|\divv\mathbf{u}\|^2_{L^2}
  +\frac{1}{2}\|\nabla\mathbf{B}\|^2_{L^2}
  -\|P\|_{L^2}\|\divv\mathbf{u}\|_{L^2}-C\|\mathbf{B}\|_{L^2}\|\nabla\mathbf{B}\|_{L^2}\|\nabla\mathbf{u}\|_{L^2}
  \notag\\
  &\geq\frac{\mu}{4}\|\nabla\mathbf{u}\|^2_{L^2}+\frac{\mu+\lambda}{4}\|\divv\mathbf{u}\|^2_{L^2}
  -\frac{\widetilde{C}}{2}\|\nabla\mathbf{B}\|^2_{L^2}
  -C(\hat{\rho})C_0^\frac{6(\alpha+1)}{7\alpha+5},
\end{align}
for some positive constant $\widetilde{C}=\widetilde{C}(\mu)$.

Integrating \eqref{3.20} with respect to $t$ over $(0,T)$, one infers from \eqref{3.7} that
\begin{align}\label{3.22}
 &\sup_{t\in[0,T]}\sigma\big(\|\nabla \mathbf{u}\|_{L^2}^2+\|\nabla \mathbf{B}\|_{L^2}^2\big)+\int_0^T\int\sigma\big(\rho |\dot{\mathbf{u}}|^2+|\Delta \mathbf{B}|^2\big)\mathrm{d}\mathbf{x}\mathrm{d}t
 \notag\\&\leq C(\hat{\rho})C_0^\frac{6(\alpha+1)}{7\alpha+5}+
C\int_0^T\int \sigma|\nabla\mathbf{u}|^3\mathrm{d}\mathbf{x}\mathrm{d}t
+C\int_0^T\int\sigma' P|\nabla\mathbf{u}|\mathrm{d}\mathbf{x}\mathrm{d}t
 +C\int_0^T\int\sigma P|\nabla\mathbf{u}|^2\mathrm{d}\mathbf{x}\mathrm{d}t\notag\\
 &\leq C(\hat{\rho})C_0^\frac{3(\alpha+1)}{7\alpha+5}+C\int_0^T\int \sigma|\nabla\mathbf{u}|^3\mathrm{d}\mathbf{x}\mathrm{d}t
 +C\int_0^T\int\sigma P|\nabla\mathbf{u}|^2\mathrm{d}\mathbf{x}\mathrm{d}t
\end{align}
provided $C_0\leq 1$, where we have used
\begin{equation*}
  \int_0^T\int\sigma' P|\nabla\mathbf{u}|\mathrm{d}\mathbf{x}\mathrm{d}t\leq
   \int_0^{\sigma(T)}\|P\|_{L^2}\|\nabla\mathbf{u}\|_{L^2}\mathrm{d}t
   \leq C(\hat{\rho})C_0^\frac{3(\alpha+1)}{7\alpha+5}\int_0^{\sigma(T)}\big(\|\nabla\mathbf{u}\|_{L^2}^2\big)^\frac12\mathrm{d}t
   \leq C(\hat{\rho})C_0^\frac{3(\alpha+1)}{7\alpha+5}.
\end{equation*}
Moreover, it follows from \eqref{3.12}, \eqref{2.4}, and H\"older's inequality that
\begin{align}\label{3.23}
  \int_0^{\sigma(T)}\int\sigma P|\nabla\mathbf{u}|^2\mathrm{d}\mathbf{x}\mathrm{d}t
  &\leq  C\int_0^{\sigma(T)}\sigma\|P\|_{L^3}\|\nabla\mathbf{u}\|_{L^2}\|\nabla\mathbf{u}\|_{L^6}\mathrm{d}t\notag\\
  &\leq C(\hat{\rho})C_0^\frac{5\alpha+7}{3(7\alpha+5)}
  \int_0^{\sigma(T)}\sigma \|\nabla\mathbf{u}\|_{L^2}\big(\|\sqrt{\rho}\|_{L^6}\|\sqrt{\rho}\dot{\mathbf{u}}\|_{L^2}+\|P\|_{L^6}+\||\mathbf{B}|^2\|_{L^6}\big)\mathrm{d}t\notag\\
  &\leq C(\hat{\rho})C_0^\frac{5\alpha+7}{3(7\alpha+5)}
   \int_0^{\sigma(T)}\sigma\big( \|\nabla\mathbf{u}\|_{L^2}^2+\|\rho\|_{L^3}\|\sqrt{\rho}\dot{\mathbf{u}}\|_{L^2}^2
   +\|P\|_{L^6}\|\nabla\mathbf{u}\|_{L^2}+\|\nabla\mathbf{B}\|_{L^2}^4\big)\mathrm{d}t\notag\\
   &\leq C(\hat{\rho})C_0^\frac{2(\alpha+2)}{7\alpha+5}+C_1(\hat{\rho})C_0^\frac{2(5\alpha+7)}{3(7\alpha+5)}A_1(T),
   \\
  \int_{\sigma(T)}^T\int\sigma P|\nabla\mathbf{u}|^2\mathrm{d}\mathbf{x}\mathrm{d}t
  &\leq  \int_{\sigma(T)}^T\|P\|_{L^4}\|\nabla\mathbf{u}\|_{L^2}\|\nabla\mathbf{u}\|_{L^4}\mathrm{d}t\notag\\
  &\leq C\int_{\sigma(T)}^T\|P\|_{L^4}\|\nabla\mathbf{u}\|_{L^2}\big(\|\sqrt{\rho}\|_{L^4}\|\sqrt{\rho}\dot{\mathbf{u}}\|_{L^2}+\|P\|_{L^4}+\||\mathbf{B}|^2\|_{L^4}\big)\mathrm{d}t\notag\\
  &\leq
   C\int_{\sigma(T)}^T\Big( \|\nabla\mathbf{u}\|_{L^2}^2+\|P\|_{L^4}^2\|\rho\|_{L^2}\|\sqrt{\rho}\dot{\mathbf{u}}\|_{L^2}^2+\|P\|_{L^4}^4
   +\|P\|_{L^4}^\frac43\|\mathbf{B}\|_{L^2}\|\nabla\mathbf{B}\|_{L^2}^4\Big)\mathrm{d}t\notag\\
   &\leq CC_0+C_2(\hat{\rho})C_0^\frac{5\alpha+7}{7\alpha+5}A_1(T)
   +\int_{\sigma(T)}^T\sigma^3\|P\|_{L^4}^4\mathrm{d}t.\label{3.24}
\end{align}
Putting \eqref{3.23} and \eqref{3.24} into \eqref{3.22} implies that
\begin{align*}
   &\sup_{t\in[0,T]}\sigma\big(\|\nabla \mathbf{u}\|_{L^2}^2+\|\nabla \mathbf{B}\|_{L^2}^2\big)+\int_0^T\int\sigma\big(\rho |\dot{\mathbf{u}}|^2+|\Delta \mathbf{B}|^2\big)\mathrm{d}\mathbf{x}\mathrm{d}t\notag\\
  &\leq
  C(\hat{\rho})C_0^\frac{2(\alpha+2)}{7\alpha+5}+2C_1(\hat{\rho})C_0^\frac{2(5\alpha+7)}{3(7\alpha+5)}A_1(T)
  +C\int_0^T\int \sigma|\nabla\mathbf{u}|^3\mathrm{d}\mathbf{x}\mathrm{d}t
+\int_{\sigma(T)}^T\sigma^3\|P\|_{L^4}^4\mathrm{d}t,
\end{align*}
which along with \eqref{3.2} immediately yields \eqref{3.10} provided
\begin{equation*}
  C_0\leq\varepsilon_1\triangleq \min\left\{1,\bigg(\frac{1}{4C_1(\hat{\rho})}\bigg)^\frac{3(7\alpha+5)}{2(5\alpha+7)}\right\}.
\end{equation*}

\textbf{Estimate for $A_2(T)$}.
Operating $\sigma^3\dot{u}^j[\partial/\partial t+\divv({\mathbf{u}}\cdot)]$ on $\eqref{3.13}^j$, summing all the equalities with respect to $j$, and integrating the resultant over $\mathbb{R}^2$,
we obtain that
\begin{align}\label{3.25}
&\frac{1}{2}\frac{\mathrm{d}}{\mathrm{d}t}\int\sigma^3\rho|\dot{\mathbf{u}}|^2\mathrm{d}\mathbf{x}
-\frac{3}{2}\sigma^2\sigma'\int\rho|\dot{\mathbf{u}}|^2\mathrm{d}\mathbf{x}\notag\\
&=-\sigma^3\int\dot{u}^j[\partial_j P_{t}+\divv(\mathbf{u}\partial_{j}P)]\mathrm{d}\mathbf{x}
+\mu\sigma^3\int\dot{u}^j\big[\Delta u_{t}^j+\divv\big( \mathbf{u}\Delta u^{j}\big)\big]\mathrm{d}\mathbf{x}\notag\\
&\quad+(\mu+\lambda)\sigma^3\int\dot{u}^j[\partial_j\divv\mathbf{u}_t +\divv(\mathbf{u}\partial_j\divv\mathbf{u})]\mathrm{d}\mathbf{x}
-\sigma^3\int\dot{u}^j\big[\partial_j(B^iB^i_t)+\divv\big(B^i\partial_jB^i\mathbf{u}\big)\big]
\mathrm{d}\mathbf{x}\notag\\
&\quad+\sigma^3\int\dot{u}^j\big[\partial_t\big(B^i\partial_iB^j\big)
+\divv\big(B^i\partial_iB^j\mathbf{u}\big)\big]
\mathrm{d}\mathbf{x}
\triangleq \sum_{i=1}^{5}\mathcal{J}_i.
\end{align}
It follows from \eqref{3.7} and Cauchy--Schwarz inequality that
\begin{align}\label{3.26}
\mathcal{J}_1&=\sigma^3\int P_{t}\divv\dot{\mathbf{u}}\mathrm{d}\mathbf{x}-\sigma^3\int\dot{\mathbf{u}}\cdot\nabla\divv(P\mathbf{u})\mathrm{d}\mathbf{x}
+\sigma^3\int\dot{u}^j\divv(P\partial_{j}\mathbf{u})\mathrm{d}\mathbf{x}\notag\\
&=\sigma^3\int \big(P_{t}+\divv(P\mathbf{u})\big)\divv\dot{\mathbf{u}}\mathrm{d}\mathbf{x}
+\sigma^3\int\dot{\mathbf{u}}\cdot\nabla\mathbf{u}\cdot\nabla P\mathrm{d}\mathbf{x}
+\sigma^3\int P\dot{\mathbf{u}}\cdot\nabla\divv\mathbf{u}\mathrm{d}\mathbf{x}\notag\\
&=-\sigma^3\int(\gamma-1)P\divv\mathbf{u}\divv\dot{\mathbf{u}}\mathrm{d}\mathbf{x}
-\sigma^3\int P\partial_i\dot{u}^j\partial_ju^i \mathrm{d}\mathbf{x}
\notag\\
&\leq \frac{\mu}{8}\sigma^3\|\nabla\dot{\mathbf{u}}\|_{L^2}^2
+C\sigma^3\|\nabla\mathbf{u}\|_{L^4}^4+C\sigma^3\|P\|_{L^4}^4.
\end{align}
Using integration by parts, one gets that
\begin{align}\label{3.27}
\mathcal{J}_2&=\mu\sigma^3\int\dot{u}^j\big[\Delta \dot{u}^j-\Delta(\mathbf{u}\cdot\nabla u^j)+\divv\big( \mathbf{u}\Delta u^{j}\big)\big]\mathrm{d}\mathbf{x}\notag\\
&=\mu\sigma^3\int\big[-|\nabla\dot{\mathbf{u}}|^2+\dot{u}^j_i(u^ku_k^j)_i
-\dot{u}^j_i(u^ku^j_i)_k-\dot{u}^j(u^k_iu^j_i)_k\big]\mathrm{d}\mathbf{x}\notag\\
&=\mu\sigma^3\int\big[-|\nabla\dot{\mathbf{u}}|^2+\dot{u}^j_i(u^ku_k^j)_i
-\dot{u}^j_i(u^ku^j_i)_k+\dot{u}_k^j(u^k_iu^j_i)\big]\mathrm{d}\mathbf{x}\notag\\
&\leq-\frac{3\mu}{4}\sigma^3\|\nabla\dot{\mathbf{u}}\|_{L^2}^2+C\sigma^3\|\nabla\mathbf{u}\|_{L^4}^4,\\
\mathcal{J}_3
&=(\mu+\lambda)\sigma^3\int\dot{u}^j[\partial_j\divv\mathbf{u}_t+ \partial_j\divv(\mathbf{u}\divv\mathbf{u})-\divv(\partial_j\mathbf{u}\divv\mathbf{u})]\mathrm{d}\mathbf{x}\notag\\
&=-(\mu+\lambda)\sigma^3\int\divv\dot{\mathbf{u}}[\divv\mathbf{u}_t+\divv(\mathbf{u}\divv\mathbf{u})]\mathrm{d}\mathbf{x}
-(\mu+\lambda)\sigma^3\int\dot{u}^j\divv(\partial_j\mathbf{u}\divv\mathbf{u})\mathrm{d}\mathbf{x}\notag\\
&=-(\mu+\lambda)\sigma^3\int\big[\divv\dot{\mathbf{u}}
\big(\divv\dot{\mathbf{u}}-\partial_iu^j\partial_ju^i+(\divv\mathbf{u})^2\big)
-\partial_j\mathbf{u}\cdot\nabla\dot{u}^j\divv\mathbf{u}\big]\mathrm{d}\mathbf{x}\notag\\
&\leq-\frac{\mu+\lambda}{2}\sigma^3\|\divv\dot{\mathbf{u}}\|_{L^2}^2
+\frac{\mu}{8}\sigma^3\|\nabla\dot{\mathbf{u}}\|_{L^2}^2+C\sigma^3\|\nabla\mathbf{u}\|_{L^4}^4.\label{3.28}
\end{align}
In addition, we deduce from $\eqref{a1}_3$ that
\begin{align}\label{3.29}
 \mathcal{J}_{4}
&=\sigma^3\int\big(\partial_j\dot{u}^jB^iB^i_t
+\partial_k\dot{u}^jB^i\partial_jB^iu^k\big)\mathrm{d}\mathbf{x}
\notag\\
&=\sigma^3\int\Big[\partial_j\dot{u}^j(\mathbf{B}\cdot\nabla\mathbf{u}
 +\nu\Delta\mathbf{B}-\mathbf{B}\divv\mathbf{u})\cdot\mathbf{B}
-\frac12\partial_k\dot{u}^j\partial_ju^k B^i B^i
+\frac12\partial_j\dot{u}^j\partial_ku^k B^i B^i\Big]\mathrm{d}\mathbf{x}
\notag\\
&\leq C\sigma^3\int|\nabla\dot{\mathbf{u}}||\nabla\mathbf{u}||\mathbf{B}|^2\mathrm{d}\mathbf{x}
+C\sigma^3\int|\nabla\dot{\mathbf{u}}||\mathbf{B}||\Delta\mathbf{B}|\mathrm{d}\mathbf{x}\notag\\
&\leq \frac{\mu\sigma^3}{16}\|\nabla \dot{\mathbf{u}}\|_{L^2}^2
+C\sigma^3\|\nabla\mathbf{u}\|_{L^4}^4+C\sigma^3\||\mathbf{B}|^2\|_{L^4}^4
+\frac{\mu}{4}\sigma^3\||\mathbf{B}||\Delta\mathbf{B}|\|_{L^2}^2,
\\
\mathcal{J}_{5}
&=-\sigma^3\int\big[\partial_i\dot{u}^j\big(B^jB^i_t+B^j_tB^i\big)
+\partial_k\dot{u}^jB^i\partial_iB^ju^k\big]\mathrm{d}\mathbf{x}
\notag\\
&\leq \frac{\mu\sigma^3}{16}\|\nabla \dot{\mathbf{u}}\|_{L^2}^2
+C\sigma^3\|\nabla\mathbf{u}\|_{L^4}^4+C\sigma^3\||\mathbf{B}|^2\|_{L^4}^4
+\frac{\mu}{4}\sigma^3\||\mathbf{B}||\Delta\mathbf{B}|\|_{L^2}^2.\label{3.30}
\end{align}
Hence, substituting \eqref{3.26}--\eqref{3.30} into \eqref{3.25}, we find that
\begin{align}\label{3.31}
\frac{\mathrm{d}}{\mathrm{d}t}\int\sigma^3\rho|\dot{\mathbf{u}}|^2\mathrm{d}\mathbf{x}
+\sigma^3\|\nabla\dot{\mathbf{u}}\|_{L^2}^2&\leq C\sigma^2\sigma'\|\sqrt{\rho}\dot{\mathbf{u}}\|_{L^2}^2+C\sigma^3\|P\|_{L^4}^4
+C\sigma^3\|\nabla\mathbf{u}\|_{L^4}^4
\notag\\&\quad
+C\sigma^3\||\mathbf{B}|^2\|_{L^4}^4
+\sigma^3\||\mathbf{B}||\Delta\mathbf{B}|\|_{L^2}^2.
\end{align}
Integrating \eqref{3.31} over $(0,T)$ implies that
\begin{align}\label{3.32}
\sup_{t\in[0,T]}\big(\sigma^3\|\sqrt{\rho}\dot{\mathbf{u}}\|_{L^2}^2\big)
 +\int_0^T\sigma^3\|\nabla\dot{\mathbf{u}}\|_{L^2}^2\mathrm{d}t
 &\leq CA_1(T)+C\int_0^T\sigma^3\|P\|_{L^4}^4\mathrm{d}t
 +C\int_0^T\sigma^3\|\nabla\mathbf{u}\|_{L^4}^4\mathrm{d}t
\notag\\
 &\quad
+CC_0\int_0^T\sigma^3\|\mathbf{B}\|_{L^4}^4\|\Delta\mathbf{B}\|_{L^2}^2\mathrm{d}t
 +\int_0^T\sigma^3\||\mathbf{B}||\Delta\mathbf{B}|\|_{L^2}^2\mathrm{d}t,
\end{align}
where we have used
\begin{equation}\label{ooz}
  \||\mathbf{B}||\nabla\mathbf{B}|\|_{L^2}^2\leq
  \|\mathbf{B}\|_{L^4}^2\|\nabla\mathbf{B}\|_{L^4}^2
  \leq C\|\mathbf{B}\|_{L^2}\|\nabla\mathbf{B}\|_{L^2}^2\|\Delta\mathbf{B}\|_{L^2}
  \leq C\|\mathbf{B}\|_{L^2}^2\|\Delta\mathbf{B}\|_{L^2}^2
   \leq CC_0\|\Delta\mathbf{B}\|_{L^2}^2.
\end{equation}

To estimate $\||\mathbf{B}||\Delta\mathbf{B}|\|_{L^2}^2$, we consider
\begin{equation*}
  \widehat{B}(a_1,a_2)\triangleq a_1B^1+a_2B^2, \ \  \widehat{u}(a_1,a_2)\triangleq a_1u^1+a_2u^2, \ \
  a_1,a_2\in\{-1,0,1\}.
\end{equation*}
Then one obtains from \eqref{a1}$_3$ that
\begin{equation}\label{3.33}
  \widehat{B}_t-\nu\Delta\widehat{B}=\mathbf{B}\cdot\nabla\widehat{u}-\mathbf{u}\cdot\nabla\widehat{B}-\widehat{B}\divv\mathbf{u}.
\end{equation}
Multiplying \eqref{3.33} by $4\widehat{B}\Delta|\widehat{B}|^2$ and integrating the resulting equation over $\mathbb{R}^2$ yields that
\begin{align*}
 \frac{\mathrm{d}}{\mathrm{d}t}\|\nabla|\widehat{B}|^2\|_{L^2}^2
 +2\nu\|\Delta|\widehat{B}|^2\|_{L^2}^2
 &\leq
 C\int\big|\Delta|\widehat{B}|^2\big|\big(|\nabla\widehat{B}|^2+|\mathbf{B}|^2|\nabla\mathbf{u}|\big)
 \mathrm{d}\mathbf{x}\notag\\
 &\leq
 C\|\nabla\mathbf{u}\|_{L^4}^4+C\|\nabla\mathbf{B}\|_{L^4}^4
 +C\||\mathbf{B}|^2\|_{L^4}^4
 +\nu\|\Delta|\widehat{B}|^2\|_{L^2}^2.
\end{align*}
Multiplying the above inequality by $\sigma^3$ and integrating the resultant over $(0,T)$ gives that
\begin{align}\label{3.34}
  &\sup_{t\in[0,T]}\big(\sigma^3\|\nabla|\widehat{B}|^2\|_{L^2}^2\big)+\int_0^T\nu\sigma^3\|\Delta|\widehat{B}|^2\|_{L^2}^2\mathrm{d}t
  \notag\\
  &\leq C\int_0^T\sigma^2\sigma'\|\nabla|\widehat{B}|^2\|_{L^2}^2\mathrm{d}t
  +C\int_0^T\sigma^3\|\nabla\mathbf{u}\|_{L^4}^4\mathrm{d}t
  +C\int_0^T\sigma^3\|\nabla\mathbf{B}\|_{L^4}^4\mathrm{d}t
  +C\int_0^T\sigma^3\||\mathbf{B}|^2\|_{L^4}^4\mathrm{d}t\notag\\
  &\leq
  C\int_0^T\sigma^3\|\nabla\mathbf{u}\|_{L^4}^4\mathrm{d}t  +C\int_0^T\sigma^3\|\nabla\mathbf{B}\|_{L^2}^2\|\Delta\mathbf{B}\|_{L^2}^2\mathrm{d}t
  +C\int_0^T\sigma^2\big(\|\mathbf{B}\|_{L^4}^4+1\big)\|\mathbf{B}\|_{L^2}^2\|\Delta\mathbf{B}\|_{L^2}^2
  \mathrm{d}t.
\end{align}
In addition, multiplying $\eqref{a1}_3$ by $4|\mathbf{B}|^2\mathbf{B}$ and integrating the resultant over $\mathbb{R}^2$ leads to
\begin{align*}
\frac{\mathrm{d}}{\mathrm{d}t}\|\mathbf{B}\|_{L^4}^4
+4\nu\||\mathbf{B}||\nabla\mathbf{B}|\|_{L^2}^2
&\leq  -2\nu\|\nabla|\mathbf{B}|^2\|_{L^2}^2+C\|\nabla\mathbf{u}\|_{L^2}\||\mathbf{B}|^2\|_{L^4}^2 \notag \\
 & \leq
 -2\nu\|\nabla|\mathbf{B}|^2\|_{L^2}^2 +C\|\nabla\mathbf{u}\|_{L^2}\||\mathbf{B}|^2\|_{L^2}\|\nabla|\mathbf{B}|^2\|_{L^2}\notag\\
  &\leq
  C\|\mathbf{B}\|_{L^4}^4\|\nabla\mathbf{u}\|_{L^2}^2,
\end{align*}
which combined with Gronwall's inequality and \eqref{3.7} implies that
\begin{align}\label{3.35}
 \sup_{t\in[0,T]}\|\mathbf{B}\|_{L^4}^4+\int_0^T\||\mathbf{B}||\nabla\mathbf{B}|\|_{L^2}^2\mathrm{d}t
 \leq C(M)C_0.
\end{align}
Note that
\begin{equation*}
  \||\mathbf{B}||\Delta\mathbf{B}|\|_{L^2}^2\leq  C\|\nabla\mathbf{B}\|_{L^4}^4
  +C\sum_{a_1,a_2}\|\Delta|\widehat{B}(a_1,a_2)|^2\|_{L^2}^2.
\end{equation*}
Taking all possible choices of $a_1$ and $a_2$, and combining \eqref{3.34} with \eqref{3.35},
we obtain that
\begin{align}\label{3.36}
  &\sup_{t\in[0,T]}\big(\sigma^3\||\mathbf{B}||\nabla\mathbf{B}|\|_{L^2}^2\big)
  +2\int_0^T\sigma^3 \||\mathbf{B}||\Delta\mathbf{B}|\|_{L^2}^2\mathrm{d}t
  \notag\\
  &\leq
C\int_0^T\sigma^3\|\nabla\mathbf{u}\|_{L^4}^4\mathrm{d}t
+C\sup_{t\in[0,T]}\big(\sigma\|\nabla \mathbf{B}\|_{L^2}^2\big)\int_0^T\sigma\|\Delta\mathbf{B}\|_{L^2}^2\mathrm{d}t
+CC_0\int_0^T\sigma\|\Delta\mathbf{B}\|_{L^2}^2\mathrm{d}t\notag\\
&\leq
CA_1(T)+CA_1^2(T)
+C\int_0^T\sigma^3\|\nabla\mathbf{u}\|_{L^4}^4\mathrm{d}t,
\end{align}
which along with \eqref{3.32} yields that
\begin{align}\label{3.37}
&\sup_{t\in[0,T]}\sigma^3\big(\|\sqrt{\rho}\dot{\mathbf{u}}\|_{L^2}^2
   +\||\mathbf{B}||\nabla \mathbf{B}|\|_{L^2}^2\big)
   +\int_0^T\sigma^3\big(\|\nabla \dot{\mathbf{u}}\|_{L^2}^2
   +\||\mathbf{B}||\Delta \mathbf{B}|\|_{L^2}^2\big)\mathrm{d}t
   \notag\\
   &\leq CA_1(T)+CA_1^2(T)+C\int_0^T\sigma^3\|P\|_{L^4}^4\mathrm{d}t
   +C\int_0^T\sigma^3\|\nabla\mathbf{u}\|_{L^4}^4\mathrm{d}t.
\end{align}

It remains to control the last term on the right-hand side of \eqref{3.37}. Following the decomposition of the velocity field in \eqref{2.6}, we get that
\begin{align}\label{3.38}
 \|\nabla\mathbf{u}\|_{L^4}^4&\leq C\|\nabla\mathbf{w}_1\|_{L^4}^4+C\|\nabla\mathbf{w}_2\|_{L^4}^4
 +C\|\nabla\mathbf{w}_3\|_{L^4}^4\notag\\
 &\leq C\|\nabla\mathbf{w}_1\|_{L^2}\|\nabla\mathbf{w}_1\|_{L^6}^3+C\|P\|_{L^4}^4
 +C\||\mathbf{B}|^2\|_{L^4}^4\notag\\
 &\leq C(\|\nabla\mathbf{u}\|_{L^2}+\|\nabla\mathbf{w}_2\|_{L^2}+\|\nabla\mathbf{w}_3\|_{L^2})\|\nabla^2\mathbf{w}_1\|_{L^{\frac32}}^3+C\|P\|_{L^4}^4
 +C\|\mathbf{B}\|_{L^4}^4\||\mathbf{B}||\nabla \mathbf{B}|\|_{L^2}^2\notag\\
 &\leq C(\|\nabla\mathbf{u}\|_{L^2}+\|P\|_{L^2}+\|\mathbf{B}\|_{L^{4}}^2)\|\rho\dot{\mathbf{u}}\|_{L^{\frac32}}^3
 +C\|P\|_{L^4}^4+C\|\mathbf{B}\|_{L^4}^4\|\mathbf{B}\|_{L^2}^2\|\Delta\mathbf{B}\|_{L^2}^2\notag\\
 &\leq C(\|\nabla\mathbf{u}\|_{L^2}+\|P\|_{L^2}+\|\mathbf{B}\|_{L^{4}}^2)\|\sqrt{\rho}\|_{L^6}^3\|\sqrt{\rho}\dot{\mathbf{u}}\|_{L^2}^3
 +C\|P\|_{L^4}^4+CC_0\|\Delta\mathbf{B}\|_{L^2}^2\notag\\
 &\leq C(\hat{\rho})(\|\nabla\mathbf{u}\|_{L^2}+1)\|\sqrt{\rho}\dot{\mathbf{u}}\|_{L^2}^3
 +C\|P\|_{L^4}^4+CC_0\|\Delta\mathbf{B}\|_{L^2}^2
\end{align}
provided $C_0\leq1$. Hence, one deduces from \eqref{3.2} and \eqref{3.3} that
\begin{align}\label{3.39}
 \int_0^T\sigma^3\|\nabla\mathbf{u}\|_{L^4}^4\mathrm{d}t
 &\leq CC_0\int_0^T\sigma^3\|\Delta\mathbf{B}\|_{L^2}^2\mathrm{d}t
 +C(\hat{\rho})
 \sup_{t\in[0,T]}\big(\sigma^\frac32\|\sqrt{\rho}\dot{\mathbf{u}}\|_{L^2}\big)\int_0^T\sigma^\frac32\|\sqrt{\rho}\dot{\mathbf{u}}\|_{L^2}^2\mathrm{d}t
 \notag\\&\quad
+C(\hat{\rho})\sup_{t\in[0,T]}\big(\sigma^\frac12\|\nabla\mathbf{u}\|_{L^2}\big)
 \sup_{t\in[0,T]}\big(\sigma^\frac32\|\sqrt{\rho}\dot{\mathbf{u}}\|_{L^2}\big)\int_0^T\sigma\|\sqrt{\rho}\dot{\mathbf{u}}\|_{L^2}^2\mathrm{d}t
 +C\int_0^T\sigma^3\|P\|_{L^4}^4\mathrm{d}t
\notag\\
 &\leq CC_0A_1(T)+C(\hat{\rho})A_2^\frac12(T)\Big(A_1(T)+A_1^\frac32(T)\Big)
 +C\int_0^T\sigma^3\|P\|_{L^4}^4\mathrm{d}t,
\end{align}
which combined with \eqref{3.37} shows \eqref{3.11}.
\end{proof}

With the preliminary bounds, we now derive the estimates for $A_1(T)$ and $A_2(T)$.

\begin{lemma}\label{l3.3}
Under the assumption \eqref{3.5}, there exists a positive constant $\varepsilon_2$ depending only on $\alpha, \hat{\rho}, M, a$, $\gamma, \mu, \lambda, \nu$, and $K$ such that
\begin{align}\label{3.40}
A_1(T)+A_2(T)\leq  C_0^\frac{\alpha+2}{7\alpha+5}
\end{align}
provided $C_0\leq\varepsilon_2$.
\end{lemma}
\begin{proof}
If $C_0\leq 1$, it follows from \eqref{3.5} and \eqref{3.7} that
\begin{align}\label{3.41}
  \int_0^T\big(\|\nabla \mathbf{u}\|_{L^2}^4+\|\nabla \mathbf{B}\|_{L^2}^4\big)\mathrm{d}t
  &\leq\sup_{t\in[0,\sigma(T)]}\big(\|\nabla \mathbf{u}\|_{L^2}^2+\|\nabla \mathbf{B}\|_{L^2}^2\big)\int_0^{\sigma(T)}\big(\|\nabla \mathbf{u}\|_{L^2}^2+\|\nabla \mathbf{B}\|_{L^2}^2\big)\mathrm{d}t\notag\\
  &\quad
  +\sup_{t\in[\sigma(T),T]}\big(\sigma\|\nabla \mathbf{u}\|_{L^2}^2+\sigma\|\nabla \mathbf{B}\|_{L^2}^2\big)\int_{\sigma(T)}^T\big(\|\nabla \mathbf{u}\|_{L^2}^2+\|\nabla \mathbf{B}\|_{L^2}^2\big)\mathrm{d}t\notag\\
  &\leq CC_0A_3(\sigma(T))+CC_0A_1(T)\notag\\
  &\leq C(K)C_0.
\end{align}
By Lemma \ref{l3.2}, we have
\begin{align}\label{3.42}
A_1(T)+A_2(T)\leq CC_0^\frac{2(\alpha+2)}{7\alpha+5}+C\int_{0}^T\sigma^3\|P\|_{L^4}^4\mathrm{d}t
+C\int_0^T \sigma\|\nabla\mathbf{u}\|_{L^3}^3\mathrm{d}t
\end{align}
provided $C_0\leq\varepsilon_1$. Using \eqref{3.41} and \eqref{2.4}, we obtain that
\begin{align}\label{3.43}
 \int_0^{\sigma(T)}\sigma\|\nabla\mathbf{u}\|_{L^3}^3\mathrm{d}t
 &\leq
 \int_0^{\sigma(T)}\sigma\|\nabla\mathbf{u}\|_{L^2}^\frac32\|\nabla\mathbf{u}\|_{L^6}^\frac32\mathrm{d}t\notag\\
 &\leq C(\hat{\rho})
 \int_0^{\sigma(T)}\sigma\|\nabla\mathbf{u}\|_{L^2}^\frac32
 \big(\|\sqrt{\rho}\|_{L^6}\|\sqrt{\rho}\dot{\mathbf{u}}\|_{L^2}+\|P\|_{L^6}+\||\mathbf{B}|^2\|_{L^6}\big)^\frac32\mathrm{d}t\notag\\
 &\leq C(\hat{\rho})\sup_{t\in[0,\sigma(T)]}\big(\sigma^\frac14\|\nabla \mathbf{u}\|_{L^2}^\frac12\big)
 \int_0^{\sigma(T)}\big(\sigma\|\sqrt{\rho}\dot{\mathbf{u}}\|_{L^2}^2+\|\nabla \mathbf{u}\|_{L^2}^4\big)\mathrm{d}t\notag\\
 &\quad+C(\hat{\rho})\sup_{t\in[0,\sigma(T)]} \|P\|_{L^6}^\frac32\int_0^{\sigma(T)}(\|\nabla\mathbf{u}\|_{L^2}^2)^\frac34\mathrm{d}t
 \notag\\
 &\quad+C(\hat{\rho})\sup_{t\in[0,\sigma(T)]}\big(\sigma^\frac12\|\nabla \mathbf{u}\|_{L^2}\big)
 \int_0^{\sigma(T)}\big(\|\nabla\mathbf{u}\|_{L^2}^2+
 \sigma\|\mathbf{B}\|_{L^2}^2\|\mathbf{B}\|_{L^4}^2\|\Delta\mathbf{B}\|_{L^2}^2\big)\mathrm{d}t
 \notag\\
 &\leq C_1(\hat{\rho},K)A_1^\frac14(T)(A_1(T)+C_0)
 +C(\hat{\rho})A_1^\frac12(T)(A_1(T)+C_0)
 +C(\hat{\rho})C_0^\frac34.
\end{align}
Moreover, one deduces from \eqref{3.39} that
\begin{align*}
  \int_{\sigma(T)}^T\sigma\|\nabla\mathbf{u}\|_{L^3}^3\mathrm{d}t
  &\leq
 \int_{\sigma(T)}^T\sigma^3\|\nabla\mathbf{u}\|_{L^2}\|\nabla\mathbf{u}\|_{L^4}^2\mathrm{d}t\notag\\
  &\leq C\int_{\sigma(T)}^T\big(\|\nabla \mathbf{u}\|_{L^2}^2+\sigma^3\|\nabla \mathbf{u}\|_{L^4}^4\big)\mathrm{d}t\notag\\
 &\leq CC_0+C(\hat{\rho})A_2^\frac12(T)\Big(A_1(T)+A_1^\frac32(T)\Big)
 +C\int_0^T\sigma^3\|P\|_{L^4}^4\mathrm{d}t,
\end{align*}
which combined with \eqref{3.42} and \eqref{3.43} implies that
\begin{align}\label{3.44}
A_1(T)+A_2(T)\leq 2C_1(\hat{\rho},K)C_0^\frac{5(\alpha+2)}{4(7\alpha+5)}
 +C\int_{0}^T\sigma^3\|P\|_{L^4}^4\mathrm{d}t.
\end{align}

To estimate the last term on the right-hand side of \eqref{3.44}, we recall \eqref{3.8}:
\begin{equation*}
P_t+\divv(P\mathbf{u})+(\gamma-1)P\divv \mathbf{u}=0.
\end{equation*}
After multiplying it by $3\sigma^3P^2$ and integrating by parts, we infer from Young's inequality that
\begin{align}\label{ef}
\frac{3\gamma-1}{2\mu+\lambda}\sigma^3\|P\|_{L^4}^4&=
-\frac{\mathrm{d}}{\mathrm{d}t}\int \sigma^3P^3\mathrm{d}\mathbf{x}
+3\sigma^2\sigma'\|P\|_{L^3}^3
-\frac{3\gamma-1}{2(2\mu+\lambda)}\int \sigma^3P^3\big(2F+|\mathbf{B}|^2\big)\mathrm{d}\mathbf{x}\notag\\
&\leq -\frac{\mathrm{d}}{\mathrm{d}t}\int \sigma^3P^3\mathrm{d}\mathbf{x}
+3\sigma^2\sigma'\|P\|_{L^3}^3
+\frac{3\gamma-1}{4(2\mu+\lambda)}\sigma^3\|P\|_{L^4}^4
+C\sigma^3\big(\|F\|_{L^4}^4+\||\mathbf{B}|^2\|_{L^4}^4\big).
\end{align}
Integrating the above inequality over $(0,T)$ and using \eqref{2.3} yields that
\begin{align}\label{3.45}
  \int_{0}^{T}\sigma^3\|P\|_{L^{4}}^{4}\mathrm{d}t&\leq C\sup_{ t\in[0, T]}\|P\|_{L^3}^3
  +C\int_{0}^{\sigma(T)}\|P\|_{L^3}^3\mathrm{d}t+\int_{0}^{T}\sigma^3\|F\|_{L^{4}}^{4}\mathrm{d}t
  +\int_{0}^{T}\sigma^3\||\mathbf{B}|^2\|_{L^{4}}^{4}\mathrm{d}t\notag\\
&\leq CC_0^\frac{5\alpha+7}{7\alpha+5}
+C\int_{0}^{T}\sigma^3\|F\|_{L^{2}}\|\nabla F\|_{L^{\frac32}}^{3}
\mathrm{d}t+CC_0\int_{0}^{T}\sigma^3\|\Delta\mathbf{B}\|_{L^{2}}^2
\mathrm{d}t
\notag\\
&\leq CC_0^\frac{5\alpha+7}{7\alpha+5}
+C\int_{0}^{T}\sigma^3\big(\|\nabla\mathbf{u}\|_{L^{2}}
+\|P\|_{L^{2}}+\|\mathbf{B}\|_{L^{4}}^2\big)
\big(\|\sqrt{\rho}\|_{L^6}^3\|\sqrt{\rho}\dot{\mathbf{u}}\|_{L^2}^3
+\|\nabla\mathbf{B}\|_{L^{2}}^5\big)
\mathrm{d}t\notag\\
&\leq CC_0^\frac{5\alpha+7}{7\alpha+5}
+C(\hat{\rho})\sup_{t\in[0,T]}\big(\sigma^\frac12\|\nabla\mathbf{u}\|_{L^{2}}+1\big)
\bigg[\sup_{t\in[0,T]}\big(\sigma^\frac32\|\sqrt{\rho}\dot{\mathbf{u}}\|_{L^2}\big)
\int_{0}^{T}\sigma\|\sqrt{\rho}\dot{\mathbf{u}}\|_{L^2}^2\mathrm{d}t\notag\\
&\quad+\sup_{t\in[0,T]}\big(\sigma^\frac32\|\nabla\mathbf{B}\|_{L^{2}}^3\big)
\int_{0}^{T}\|\nabla\mathbf{B}\|_{L^{2}}^2\mathrm{d}t
\bigg]\notag\\
&\leq  CC_0^\frac{5\alpha+7}{7\alpha+5}+
C_3(\hat{\rho})\Big(A_1^\frac12(T)+1\Big)
\Big(A_1(T)A_2^\frac12(T)+A_1^\frac32(T)\Big).
\end{align}

As a consequence, one deduces from \eqref{3.44} and \eqref{3.45} that
\begin{align*}
  A_1(T)+A_2(T)&\leq 2C_1(\hat{\rho},K)C_0^\frac{5(\alpha+2)}{4(7\alpha+5)}
 +C_3(\hat{\rho})\Big(A_1^\frac12(T)+1\Big)
 \Big(A_1(T)A_2^\frac12(T)+A_1^\frac32(T)\Big)\notag\\
 &\leq 2C_1(\hat{\rho},K)C_0^\frac{\alpha+2}{4(7\alpha+5)}
 C_0^\frac{\alpha+2}{7\alpha+5}
 +4C_3(\hat{\rho})C_0^\frac{\alpha+2}{2(7\alpha+5)}
 C_0^\frac{\alpha+2}{7\alpha+5},
\end{align*}
which shows \eqref{3.40} if we take
\begin{equation*}
  C_0\leq\varepsilon_2\triangleq \min\left\{\varepsilon_1,
  \bigg(\frac{1}{4C_1(\hat{\rho},K)}\bigg)^\frac{4(7\alpha+5)}{\alpha+2},
  \bigg(\frac{1}{8C_3(\hat{\rho})}\bigg)^\frac{2(7\alpha+5)}{\alpha+2}\right\}.\tag*{\qedhere}
\end{equation*}
\end{proof}

Next, we bound the auxiliary functional $A_3(T)$ and $\sigma\|\sqrt{\rho}\dot{\mathbf{u}}\|_{L^2}^2$
for small time.

\begin{lemma}\label{l3.4}
Under the assumption \eqref{3.5}, there exist positive constants $K\geq 2M^2$, $C(K)$, and $\varepsilon_3$ which depend only on $\alpha, \hat{\rho}, M, a, \gamma, \mu, \lambda$, and $\nu$ such that
\begin{gather}\label{3.46}
  A_3(\sigma(T))=\sup_{t\in[0,\sigma(T)]}\big(\|\nabla \mathbf{u}\|_{L^2}^2+\|\nabla \mathbf{B}\|_{L^2}^2\big)
  +\int_0^{\sigma(T)}\big(\|\sqrt{\rho}\dot{\mathbf{u}}\|_{L^2}^2
  +\|\Delta\mathbf{B}\|_{L^2}^2\big)\mathrm{d}t
  \leq 2K, \\
  \sup_{t\in[0,\sigma(T)]}\big(\sigma\|\sqrt{\rho} \dot{\mathbf{u}}\|_{L^2}^2\big)
   +\int_0^{\sigma(T)}\sigma\|\nabla \dot{\mathbf{u}}\|_{L^2}^2\mathrm{d}t\leq C(K)\label{3.47}
\end{gather}
provided $C_0\leq\varepsilon_3$.
\end{lemma}
\begin{proof}
Multiplying \eqref{3.13} by $\dot{\mathbf{u}}$ and integrating the resultant over $\mathbb{R}^2$, one has that
\begin{align*}
\int\rho |\dot{\mathbf{u}}|^2\mathrm{d}\mathbf{x}&=
\int[-\dot{\mathbf{u}}\cdot\nabla P+\mu\Delta\mathbf{u}\cdot\dot{\mathbf{u}}
+(\mu+\lambda)\dot{\mathbf{u}}\cdot\nabla\divv \mathbf{u}]\mathrm{d}\mathbf{x}
-\frac12\int\dot{\mathbf{u}}\cdot\nabla|\mathbf{B}|^2
\mathrm{d}\mathbf{x}
+\int\mathbf{B}\cdot\nabla\mathbf{B}\cdot\dot{\mathbf{u}}\mathrm{d}\mathbf{x}.
\end{align*}
Following the same arguments as those in \eqref{3.15}--\eqref{3.20}, with $\sigma$ replaced by $1$, we obtain from \eqref{1.10} and \eqref{3.7} that
\begin{align}\label{3.48}
  &\sup_{t\in[0,\sigma(T)]}\big(\|\nabla \mathbf{u}\|_{L^2}^2+\|\nabla \mathbf{B}\|_{L^2}^2\big)
  +\int_0^{\sigma(T)}\big(\|\sqrt{\rho}\dot{\mathbf{u}}\|_{L^2}^2
  +\|\Delta\mathbf{B}\|_{L^2}^2\big)\mathrm{d}t\notag\\
  &\leq CM^2
+C\int_0^{\sigma(T)}\|\nabla\mathbf{u}\|^3_{L^3}\mathrm{d}t
+C\int_0^{\sigma(T)}\int P|\nabla\mathbf{u}|^2\mathrm{d}\mathbf{x}\mathrm{d}t
\end{align}
provided $C_0\leq 1$.

It remains to estimate the last two terms on the right-hand side of \eqref{3.48}. Using \eqref{2.4} and \eqref{3.12}, we find that
\begin{align}\label{3.49}
 \int_0^{\sigma(T)}\|\nabla\mathbf{u}\|_{L^3}^3\mathrm{d}t
 &\leq
 \int_0^{\sigma(T)}\|\nabla\mathbf{u}\|_{L^2}^\frac32\|\nabla\mathbf{u}\|_{L^6}^\frac32\mathrm{d}t\notag\\
 &\leq C
 \int_0^{\sigma(T)}\|\nabla\mathbf{u}\|_{L^2}^\frac32
 \big(\|\sqrt{\rho}\|_{L^6}\|\sqrt{\rho}\dot{\mathbf{u}}\|_{L^2}+\|P\|_{L^6}+\||\mathbf{B}|^2\|_{L^6}\big)^\frac32\mathrm{d}t\notag\\
 &\leq C(\hat{\rho})\sup_{t\in[0,\sigma(T)]}\big(\|\nabla \mathbf{u}\|_{L^2}^2\big)^\frac34
 \sup_{t\in[0,\sigma(T)]} \|\rho\|_{L^3}^\frac34
 \int_0^{\sigma(T)}\big(\|\sqrt{\rho}\dot{\mathbf{u}}\|_{L^2}^2\big)^\frac34\mathrm{d}t\notag\\
 &\quad+C(\hat{\rho})\sup_{t\in[0,\sigma(T)]} \|P\|_{L^6}^\frac32\int_0^{\sigma(T)}(\|\nabla\mathbf{u}\|_{L^2}^2)^\frac34\mathrm{d}t\notag\\
 &\quad+C(\hat{\rho})C_0^\frac34\sup_{t\in[0,\sigma(T)]}\big(\|\nabla \mathbf{u}\|_{L^2}^2\big)^\frac34
 \sup_{t\in[0,\sigma(T)]}\big(\|\nabla \mathbf{B}\|_{L^2}^2\big)^\frac14
 \int_0^{\sigma(T)}\big(\|\Delta\mathbf{B}\|_{L^2}^2\big)^\frac12\mathrm{d}t\notag\\
 &\leq C_4(\hat{\rho})C_0^\frac{5\alpha+7}{4(7\alpha+5)}
 A_3^\frac32(\sigma(T))+C(\hat{\rho})C_0^\frac34, 
 \\
 \int_0^{\sigma(T)}\int P|\nabla\mathbf{u}|^2\mathrm{d}\mathbf{x}\mathrm{d}t
  &\leq  \int_0^{\sigma(T)}\|P\|_{L^3}\|\nabla\mathbf{u}\|_{L^2}\|\nabla\mathbf{u}\|_{L^6}\mathrm{d}t\notag\\
  &\leq C(\hat{\rho})C_0^\frac{5\alpha+7}{3(7\alpha+5)}
  \int_0^{\sigma(T)} \|\nabla\mathbf{u}\|_{L^2}\big(\|\sqrt{\rho}\|_{L^6}\|\sqrt{\rho}\dot{\mathbf{u}}\|_{L^2}+\|P\|_{L^6}+\||\mathbf{B}|^2\|_{L^6}\big)\mathrm{d}t\notag\\
 &\leq C(\hat{\rho})C_0^\frac{5\alpha+7}{2(7\alpha+5)}
 A_3(\sigma(T))+C(\hat{\rho})C_0^\frac12.\label{3.50}
\end{align}
Thus, it follows from \eqref{3.48}--\eqref{3.50} that
\begin{align*}
   A_3(\sigma(T))
 &\leq C(\hat{\rho},M)+C_4(\hat{\rho})C_0^\frac{5\alpha+7}{4(7\alpha+5)}
 A_3^\frac32(\sigma(T))+C(\hat{\rho})C_0^\frac{5\alpha+7}{2(7\alpha+5)}
 A_3(\sigma(T))\notag\\
&\leq K+2C_4(\hat{\rho})C_0^\frac{5\alpha+7}{4(7\alpha+5)}
 A_3^\frac32(\sigma(T))
\end{align*}
for $K\triangleq C(\hat{\rho},M)$, which immediately yields \eqref{3.46} provided
\begin{equation*}
  C_0\leq\varepsilon_{3}\triangleq \min\left\{\varepsilon_2,
  \left(\frac{1}{18KC_4(\hat{\rho})}\right)^\frac{4(7\alpha+5)}{5\alpha+7}\right\}.
\end{equation*}

We now proceed to prove \eqref{3.47}. Operating $\sigma\dot{u}^j[\partial/\partial t+\divv({\mathbf{u}}\cdot)]$ on $\eqref{3.13}^j$, summing all the equalities with respect to $j$, and integrating the resultant over $\mathbb{R}^2\times(0,\sigma(T))$, we deduce from \eqref{3.34} and \eqref{3.25}--\eqref{3.31} (with $\sigma^3$ replaced by $\sigma$) that
\begin{align}\label{3.51}
&\sup_{t\in[0,\sigma(T)]}\big(\sigma\|\sqrt{\rho}\dot{\mathbf{u}}\|_{L^2}^2\big)
 +\int_0^{\sigma(T)}\sigma\|\nabla\dot{\mathbf{u}}\|_{L^2}^2\mathrm{d}t\notag\\
 &\leq C\int_0^{\sigma(T)}\|\sqrt{\rho}\dot{\mathbf{u}}\|_{L^2}^2\mathrm{d}t
 +CC_0\int_0^{\sigma(T)}\sigma\|\mathbf{B}\|_{L^4}^4\|\Delta\mathbf{B}\|_{L^2}^2\mathrm{d}t
 +C\int_0^{\sigma(T)}\sigma\big(\|\nabla\mathbf{u}\|_{L^4}^4+\|P\|_{L^4}^4\big)\mathrm{d}t
 +CK\notag\\
 &\leq C(K)
 +C(\hat{\rho})\sup_{t\in[0,\sigma(T)]}\big(\sigma^\frac12\|\nabla\mathbf{u}\|_{L^{2}}+\sigma^\frac12\|\nabla\mathbf{B}\|_{L^{2}}
 +\|P\|_{L^2}+1\big)
\notag\\
&\quad\times\sup_{t\in[0,\sigma(T)]}\big(\sigma^\frac12\|\sqrt{\rho}\dot{\mathbf{u}}\|_{L^2}\big)
\int_{0}^{\sigma(T)}\big(\|\sqrt{\rho}\dot{\mathbf{u}}\|_{L^2}^2
+\|\nabla\mathbf{B}\|_{L^{2}}^4\big)\mathrm{d}t\notag\\
 &\leq C(K)+C(\hat{\rho})\big(A_3(\sigma(T))+C_0\big)
\sup_{t\in[0,\sigma(T)]}\big(\sigma\|\sqrt{\rho}\dot{\mathbf{u}}\|_{L^2}^2\big)^\frac12,
\end{align}
owing to \eqref{3.38} and  \eqref{3.41},
which along with \eqref{3.46} indicates \eqref{3.47}.
\end{proof}

To bound $\|\rho\|_{L^\theta}$, we will employ the following spatial weighted estimate.

\begin{lemma}\label{l3.5}
Let $\bar{x}$ and $\alpha\in(1,2)$ be as in \eqref{1.7}.
There exists a positive constant $C$ depending only on $\alpha, \hat{\rho}, M, a,
\gamma, \mu, \lambda$, and $\eta_0$ such that, for any $0<t<T$,
\begin{equation}\label{3.52}
  \sup_{s\in[0,t]}\|\bar{x}^\alpha\rho\|_{L^1}\leq C(1+t)^4.
\end{equation}
\end{lemma}
\begin{proof}
Multiplying $\eqref{a1}_1$ by $(1+|\mathbf{x}|^2)^\frac12$ and integrating the resulting equality over $\mathbb{R}^2$, we obtain that
\begin{equation*}
  \frac{\mathrm{d}}{\mathrm{d}t}\int\rho(1+|\mathbf{x}|^2)^\frac12\mathrm{d}\mathbf{x}
\leq C\int\rho|\mathbf{u}|\mathrm{d}\mathbf{x}
\leq C\|\sqrt{\rho}\|_{L^2}\|\sqrt{\rho}\mathbf{u}\|_{L^2}.
\end{equation*}
Integrating the above inequality with respect to time over $(0,t)$ together with \eqref{3.1} and \eqref{3.7} yields that
\begin{equation}\label{3.53}
  \sup_{s\in[0,t]}\int\rho(1+|\mathbf{x}|^2)^\frac12\mathrm{d}\mathbf{x}\leq C(M)(1+t).
\end{equation}

For $N\geq1$, let $\psi_N\in C^\infty_0(B_{2N})$ satisfy
\begin{equation}\label{3.54}
  \psi_N(\mathbf{x})=
  \begin{cases}
  1, \ \ |\mathbf{x}|\leq N, \\
  0, \ \ |\mathbf{x}|\geq 2N,
  \end{cases}
  \ \ \ 0\leq\psi_N\leq1, \ \ \ |\nabla\psi_N|\leq CN^{-1}.
\end{equation}
Set
\begin{equation*}
\mathbf{y}(t)=\delta\mathbf{x}(1+t)^{-1}\log^{-\alpha}(e+t),
\end{equation*}
with some small constant $\delta>0$ determined later. Multiplying $\eqref{a1}_1$ by $\psi_1(\mathbf{y})$ and using integration by parts, one gets from \eqref{3.53} that
\begin{align}\label{3.55}
 \frac{\mathrm{d}}{\mathrm{d}t}\int\rho\psi_{1}(\mathbf{y})\mathrm{d}\mathbf{x}
 &=\int\rho\mathbf{y}_t\cdot\nabla_{\mathbf{y}}\psi_1\mathrm{d}\mathbf{x}
 +\delta(1+t)^{-1}\log^{-\alpha}(e+t)\int\rho\mathbf{u}\cdot\nabla_{\mathbf{y}}\psi_1\mathrm{d}\mathbf{x}\notag\\
 &\geq -\frac{C\delta}{(1+t)^{2}\log^{\alpha}(e+t)}\int\rho|\mathbf{x}|\mathrm{d}\mathbf{x}
 -\frac{C\delta}{(1+t)\log^{\alpha}(e+t)}\notag\\
&\geq -\frac{2C(M)\delta}{(1+t)\log^{\alpha}(e+t)}.
\end{align}
It thus follows from \eqref{1.9} and \eqref{3.54} that
\begin{equation*}
  \int\rho\psi_{1}(\mathbf{y})\mathrm{d}\mathbf{x}
  \geq \int\rho_0\psi_{1}(\delta\mathbf{x})\mathrm{d}\mathbf{x}-C(\alpha,M)\delta
  \geq \int_{B_{\eta_0}}\rho_0\mathrm{d}\mathbf{x}-C(\alpha,M)\delta
  \geq \frac{1}{4}
\end{equation*}
by selecting $\delta\triangleq\big(\eta_0+4C(\alpha,M)\big)^{-1}$. Furthermore,
for $\eta_1\triangleq 2\delta^{-1}=2\eta_0+8C(\alpha,M)$, we deduce that
\begin{equation}\label{3.56}
\inf_{0\leq t\leq T}\int_{B_{\eta_1(1+t)\log^{\alpha}(e+t)}}\rho\mathrm{d}\mathbf{x}\geq
\inf_{0\leq t\leq T}\int\rho\psi_{1}(\mathbf{y})\mathrm{d}\mathbf{x}\geq \frac{1}{4},
\end{equation}
as the desired \eqref{1.13}.

Now we use Lemma \ref{l2.4} with the radius $\eta_*=\eta_1(1+t)\log^{\alpha}(e+t)$.
It follows from  Lemma \ref{l2.3}, \eqref{2.1}, \eqref{3.1},  and \eqref{3.53} that
\begin{equation}\label{3.57}
  \sup_{s\in[0,t]}\|\bar{x}^{-\varsigma}\mathbf{u}\|_{L^{r/\varsigma}}\leq
  C\sup_{s\in[0,t]}\eta_*^2(1+\|\rho\|_{L^2})\big(\|\sqrt{\rho}\mathbf{u}\|_{L^2}+\|\nabla\mathbf{u}\|_{L^2}\big)
  \leq C(\hat{\rho},M)(1+t)^{3-\frac{\alpha+1}{3}}
\end{equation}
for $\varsigma\in(0,1]$, $\alpha\in(1,2)$, and $r>2$.
At last, multiplying $\eqref{a1}_1$ by $\bar{x}^\alpha$ and integrating the resultant over $\mathbb{R}^2$, one gets from H{\"o}lder's inequality, \eqref{3.53}, and \eqref{3.57} that
\begin{align*}
  \frac{\mathrm{d}}{\mathrm{d}t}\int\bar{x}^\alpha\rho\mathrm{d}\mathbf{x}
  &\leq C\int\rho|\mathbf{u}|\bar{x}^{\alpha-1}\log^2(e+|\mathbf{x}|^2)\mathrm{d}\mathbf{x}\notag\\
  &\leq C\big\|\big(\rho(1+|\mathbf{x}|^2)^\frac12\big)^\frac{\alpha+1}{3}\big\|_{L^\frac{3}{\alpha+1}}
  \big\|\bar{x}^{-\frac{2-\alpha}{2}}\mathbf{u}\big\|_{L^{\frac{6}{2-\alpha}}}
  \big\|\rho^{\frac{2-\alpha}{3}}\big\|_{L^{\frac{6}{2-\alpha}}}\notag\\
  &\quad\times\sup_{\mathbf{x}\in\mathbb{R}^2}
  \Big[(1+|\mathbf{x}|^2)^{\frac{\alpha-2}{12}}\log^{\alpha+2}(e+|\mathbf{x}|^2)\Big]
  \notag\\
  &\leq C(1+t)^3,
\end{align*}
which yields \eqref{3.52} after integrating over $(0,t)$.
\end{proof}

Finally, by Zlotnik's inequality, one can establish the desired uniform-in-time $L^\theta$ bound for the density.

\begin{lemma}\label{l3.6}
Under the assumption \eqref{3.5}, there exists a positive constant $\varepsilon_4$ depending only on $\alpha, \hat{\rho}, M, a$, $\gamma, \mu, \lambda, \nu$, and $\eta_0$ such that
\begin{align}\label{3.58}
\sup_{t\in[0,T]}\|\rho\|_{L^\theta}\leq\frac{7}{4}\hat{\rho}
\end{align}
provided $C_0\leq \varepsilon_4$.
\end{lemma}
\begin{proof}
From $\eqref{a1}_1$ and \eqref{1.5}, we derive that
\begin{equation*}
\frac{\mathrm{d}}{\mathrm{d}t}\int\rho^\theta\mathrm{d}\mathbf{x}
=
-\frac{\theta-1}{2\mu+\lambda}\int\rho^\theta\Big(P+F+\frac12|\mathbf{B}|^2\Big)\mathrm{d}\mathbf{x}\leq
-\frac{a(\theta-1)}{2\mu+\lambda}\|\rho\|_{L^{\theta+\gamma}}^{\theta+\gamma}
+\frac{\theta-1}{2\mu+\lambda}\|\rho\|_{L^\theta}^\theta\big(\|F\|_{L^\infty}+\|\mathbf{B}\|_{L^\infty}^2\big).
\end{equation*}
Noting that
\begin{equation*}
\|\rho\|_{L^\theta}^\theta\leq \|\rho\|_{L^\gamma}^{\frac{\gamma^2}{\theta}}
\|\rho\|_{L^{\theta+\gamma}}^{\frac{\theta^2-\gamma^2}{\theta}}\leq \big(a^{-1}(\gamma-1)C_0\big)^\frac{\gamma}{\theta}\|\rho\|_{L^{\theta+\gamma}}^{\frac{\theta^2-\gamma^2}{\theta}}
\leq a^{-\frac{\gamma}{\theta}}(\gamma-1)^\frac{\gamma}{\theta}\|\rho\|_{L^{\theta+\gamma}}^{\frac{\theta^2-\gamma^2}{\theta}}
\end{equation*}
provided $C_0\leq1$, one deduces that
\begin{equation}\label{3.59}
   \frac{\mathrm{d}}{\mathrm{d}t}\|\rho\|_{L^\theta}^\theta\leq
   -\frac{a\tilde{c}(\theta-1)}{2\mu+\lambda}
   \big(\|\rho\|_{L^{\theta}}^{\theta}\big)^\frac{\theta}{\theta-\gamma}
   +\frac{\theta-1}{2\mu+\lambda}\|\rho\|_{L^\theta}^\theta\big(\|F\|_{L^\infty}+\|\mathbf{B}\|_{L^\infty}^2\big),
\end{equation}
where $\tilde{c}=\tilde{c}(\alpha,a,\gamma)=a^{-\frac{\theta^2}{\gamma(\theta-\gamma)}}
(\gamma-1)^{\frac{\theta^2}{\gamma(\theta-\gamma)}}$.

Set
\begin{align*}
 y(t)=\|\rho\|_{L^\theta}^\theta,\ \
  f(y)=-\frac{a\widetilde{c}(\theta-1)}{2\mu+\lambda}y^\frac{\theta}{\theta-\gamma}(t),\ \
  b(t)=\int_0^t\frac{\theta-1}{2\mu+\lambda}
  \big(\|F\|_{L^\infty}+\|\mathbf{B}\|_{L^\infty}^2\big)y(\tau)\mathrm{d}\tau.
\end{align*}
It follows that
\begin{equation}\label{3.60}
y'(t)\leq f(y)+b'(t).
\end{equation}
Using Lemma \ref{l2.4} with $\eta_*=\eta_1(1+t)\log^{\alpha}(e+t)$, we get from \eqref{2.2}, \eqref{3.12}, \eqref{3.52}, and \eqref{3.56} that
\begin{align}\label{3.61}
  \|\rho\dot{\mathbf{u}}\|_{L^4}
   \leq
  C(\hat{\rho},M)(1+t)^4\big(\|\sqrt{\rho}\dot{\mathbf{u}}\|_{L^2}
  +\|\nabla\dot{\mathbf{u}}\|_{L^2} \big).
\end{align}
It follows from \eqref{3.47}, Lemma \ref{l2.2}, and Lemma \ref{l2.6} that, for $0\leq t_1<t_2\leq \sigma(T)$,
\begin{align}\label{3.62}
&|b(t_2)-b(t_1)|
\notag\\
&\leq C
 \int_0^{\sigma(T)}\big(\|F\|_{L^\infty}+\|\mathbf{B}\|_{L^\infty}^2\big)\mathrm{d}t\notag\\
 &\leq C\int_0^{\sigma(T)}\|F\|_{L^2}^{\frac13}\|\nabla F\|_{L^4}^{\frac23}\mathrm{d}t
 +C\int_0^{\sigma(T)}\|\mathbf{B}\|_{L^4}\|\nabla \mathbf{B}\|_{L^2}^{\frac12}\|\Delta \mathbf{B}\|_{L^2}^{\frac12}\mathrm{d}t
 \notag\\
 &\leq C\int_0^{\sigma(T)}\big(\|\nabla\mathbf{u}\|_{L^{2}}+\|P\|_{L^{2}}
 +\|\mathbf{B}\|_{L^{4}}^2\big)^\frac13
 \big(\|\rho\dot{\mathbf{u}}\|_{L^4}+\||\mathbf{B}||\nabla\mathbf{B}|\|_{L^4}\big)^\frac23\mathrm{d}t
 +C\int_0^{\sigma(T)}\|\nabla \mathbf{B}\|_{L^2}^{\frac12}\|\Delta \mathbf{B}\|_{L^2}^{\frac12}\mathrm{d}t\notag\\
 &\leq C\int_0^{\sigma(T)}\big(\|\nabla\mathbf{u}\|_{L^{2}}^\frac13+\|P\|_{L^{2}}^\frac13 +\|\mathbf{B}\|_{L^{4}}^\frac23\big)
 \Big(\|\sqrt{\rho}\dot{\mathbf{u}}\|_{L^2}^\frac23+\|\nabla\dot{\mathbf{u}}\|_{L^2}^\frac23
 +\|\Delta\mathbf{B}\|_{L^2}^\frac23+\||\mathbf{B}||\Delta\mathbf{B}|\|_{L^2}^\frac13
  \Big)\mathrm{d}t\notag\\
  &\quad +C\int_0^{\sigma(T)}\|\nabla \mathbf{B}\|_{L^2}^{\frac12}\|\Delta \mathbf{B}\|_{L^2}^{\frac12}\mathrm{d}t\notag\\
  &\leq
   C\sup_{t\in[0,\sigma(T)]}
   \Big(\sigma^\frac16\|\nabla\mathbf{u}\|_{L^{2}}^\frac13\Big)\times
  \Bigg[\bigg(\int_{0}^{\sigma(T)}\big(\|\sqrt{\rho}\dot{\mathbf{u}}\|_{L^{2}}^{2}+\|\Delta\mathbf{B}\|_{L^2}^2\big)
  \mathrm{d}t\bigg)^{\frac{1}{3}}\bigg(\int_{0}^{\sigma(T)}\sigma^{-\frac14}\mathrm{d}t\bigg)^{\frac{2}{3}}
\notag\\&\quad+\bigg(\int_0^{\sigma(T)}\sigma\|\nabla\dot{\mathbf{u}}\|_{L^2}^2
\mathrm{d}t\bigg)^{\frac13}\bigg(\int_0^{\sigma(T)}
\sigma^{-\frac34}\mathrm{d}t\bigg)^{\frac23}
+\bigg(\int_0^{\sigma(T)}\sigma^3\||\mathbf{B}||\Delta\mathbf{B}|\|_{L^2}^2
\mathrm{d}t\bigg)^{\frac16}\bigg(\int_0^{\sigma(T)}
\sigma^{-\frac45}\mathrm{d}t\bigg)^{\frac56}\Bigg]\notag\\
  &\quad+C\sup_{t\in[0,\sigma(T)]}\|P\|_{L^{2}}^\frac13\times
  \Bigg[\bigg(\int_{0}^{\sigma(T)}\big(\|\sqrt{\rho}\dot{\mathbf{u}}\|_{L^{2}}^{2}+\|\Delta\mathbf{B}\|_{L^2}^2\big)\mathrm{d}t\bigg)^{\frac{1}{3}}\bigg(\int_{0}^{\sigma(T)}1\mathrm{d}t\bigg)^{\frac{2}{3}}
\notag\\&\quad+\bigg(\int_0^{\sigma(T)}\sigma\|\nabla\dot{\mathbf{u}}\|_{L^2}^2
\mathrm{d}t\bigg)^{\frac13}\bigg(\int_0^{\sigma(T)}
\sigma^{-\frac12}\mathrm{d}t\bigg)^{\frac23}+\bigg(\int_0^{\sigma(T)}\sigma^3\||\mathbf{B}||\Delta\mathbf{B}|\|_{L^2}^2
\mathrm{d}t\bigg)^{\frac16}\bigg(\int_0^{\sigma(T)}
\sigma^{-\frac35}\mathrm{d}t\bigg)^{\frac56}\Bigg]\notag\\
&\quad + C\sup_{t\in[0,\sigma(T)]}
   \Big(\sigma^\frac14\|\nabla\mathbf{B}\|_{L^{2}}^\frac12\Big)\times
  \bigg(\int_{0}^{\sigma(T)}\sigma\|\Delta\mathbf{B}\|_{L^{2}}^{2}\mathrm{d}t\bigg)^\frac14
\bigg(\int_0^{\sigma(T)}\sigma^{-\frac23}\mathrm{d}t\bigg)^{\frac34}\notag\\
&\leq C_1(\hat{\rho},M,K)C_0^\frac{\alpha+2}{6(7\alpha+5)}.
\end{align}
By applying Lemma \ref{lzlo} with
\begin{equation*}
  N_1=0, \ \
  N_0=C_1(\hat{\rho},M,K)C_0^\frac{\alpha+2}{6(7\alpha+5)}, \ \
  \xi^*=1,
\end{equation*}
one sees that
\begin{equation*}
  f(\xi)=-\frac{a\tilde{c}(\theta-1)}{2\mu+\lambda}\xi^\frac{\theta}{\theta-\gamma}\leq-N_1=0,
  \ \ \ \text{for all} \ \xi\geq \xi^*=1.
\end{equation*}
Hence, it holds that
\begin{equation}\label{3.63}
  y(t)\leq \max \big\{\hat{\rho}^\theta, 1\big\} + N_{0}\leq\hat{\rho}^\theta+C_1(\hat{\rho},M,K)C_0^\frac{\alpha+2}{6(7\alpha+5)}
  \leq \Big(\frac{3}{2}\hat{\rho}\Big)^\theta
\end{equation}
provided
\begin{equation*}
  C_0\leq \varepsilon_{4,1}\triangleq
  \min\left\{\varepsilon_3,
  \bigg(\frac{\hat{\rho}^\theta}{C_1(\hat{\rho},M,K)}\bigg)
  ^\frac{6(7\alpha+5)}{\alpha+2}\right\}.
\end{equation*}

For $\sigma(T)\leq t_1<t_2\leq T$, we first establish large-time weighted estimates for the material derivative.
We start from the equation \eqref{3.8} satisfied by the pressure:
\begin{equation*}
  P_t+\divv(P\mathbf{u})+(\gamma-1)P\divv \mathbf{u}=0.
\end{equation*}
Testing the above equality with $q(tP)^{q-1} \ (q\geq2)$ yields
\begin{align*}
   &\frac{\mathrm{d}}{\mathrm{d}t}\big(t^{q-1}\|P\|_{L^q}^q\big)
 +\frac{q\gamma-1}{2\mu+\lambda}t^{q-1}\|P\|_{L^{q+1}}^{q+1}
 \notag\\
 &\leq
 Ct^{q-2}\|P\|_{L^q}^q
 -\frac{q\gamma-1}{2(2\mu+\lambda)}\int t^{q-1}P^q(2F+|\mathbf{B}|^2)\mathrm{d}\mathbf{x}
 \notag\\
 &\leq
  Ct^{q-2}\|P\|_{L^q}^q
  +\frac{q\gamma-1}{2(2\mu+\lambda)}t^{q-1}\|P\|_{L^{q+1}}^{q+1}
  +Ct^{q-1}\Big(\|F\|_{L^{q+1}}^{q+1}+\||\mathbf{B}|^2\|_{L^{q+1}}^{q+1}\Big),
\end{align*}
and hence
\begin{align}\label{ef1}
   \frac{\mathrm{d}}{\mathrm{d}t}\big(t^{q-1}\|P\|_{L^q}^q\big)
 +t^{q-1}\|P\|_{L^{q+1}}^{q+1}
\leq
  Ct^{q-2}\|P\|_{L^q}^q
  +Ct^{q-1}\Big(\|F\|_{L^{q+1}}^{q+1}+\||\mathbf{B}|^2\|_{L^{q+1}}^{q+1}\Big).
\end{align}
In particular, taking $q=2$ in \eqref{ef1}, one gets from \eqref{3.40} and \eqref{3.45} that
\begin{align}\label{ef2}
   &\frac{\mathrm{d}}{\mathrm{d}t}\big(t\|P\|_{L^2}^2\big)
 +t\|P\|_{L^{3}}^{3}
 \notag\\
&\leq
  C\|P\|_{L^2}^2
  +Ct\big(\|F\|_{L^{3}}^{3}+\||\mathbf{B}|^2\|_{L^{3}}^{3}\big)
  \notag\\
  &\leq
  C\|P\|_{L^2}^2+\frac{1}{8\widehat{C}}t\|\sqrt{\rho}\dot{\mathbf{u}}\|_{L^2}^2
  +Ct(1+A_1(T))\big(\|\nabla\mathbf{u}\|_{L^{2}}+\|\nabla\mathbf{B}\|_{L^{2}}+\|P\|_{L^{2}}\big)^4
  \notag\\
  &\leq
  C\|P\|_{L^2}^2+\frac{1}{8\widehat{C}}t\|\sqrt{\rho}\dot{\mathbf{u}}\|_{L^2}^2
  +Ct\big(\|\nabla\mathbf{u}\|_{L^{2}}+\|\nabla\mathbf{B}\|_{L^{2}}+\|P\|_{L^{2}}\big)^4,
  \quad \text{for } t\in (\sigma(T),T),
\end{align}
owing to
\begin{align*}
C\|F\|_{L^{3}}^{3}
\leq C \|F\|_{L^2}^\frac32\|\nabla F\|_{L^\frac32}^\frac32
&\leq
C\big(\|\nabla\mathbf{u}\|_{L^{2}}+\|P\|_{L^{2}}+\|\mathbf{B}\|_{L^{4}}^2\big)^\frac32
\big(\|\sqrt{\rho}\|_{L^6}\|\sqrt{\rho}\dot{\mathbf{u}}\|_{L^2}+\||\mathbf{B}||\nabla\mathbf{B}|\|_{L^\frac32}\big)^\frac32
\notag\\
&\leq
C\big(\|\nabla\mathbf{u}\|_{L^{2}}+\|\nabla\mathbf{B}\|_{L^{2}}+\|P\|_{L^{2}}\big)^\frac32
\big(\|\sqrt{\rho}\dot{\mathbf{u}}\|_{L^2}+\|\nabla\mathbf{B}\|_{L^{2}}^\frac53\big)^\frac32.
\end{align*}
Arguing as in \eqref{3.20}, but with $\sigma$ replaced by $t$, and then adding $\widehat{C}\times\eqref{ef2}$, we obtain
\begin{align}\label{ef3}
&\frac{\mathrm{d}}{\mathrm{d}t}\big(t\mathcal{E}_1(t)+\widetilde{C}t\|\nabla\mathbf{B}\|^2_{L^2}
+\widehat{C}t\|P\|_{L^2}^2\big)+
\frac14t\big(\|\sqrt{\rho}\dot{\mathbf{u}}\|_{L^2}^2+\nu\|\Delta\mathbf{B}\|_{L^2}^2+\widehat{C}\|P\|_{L^{3}}^{3}\big)
\notag\\
&\leq
 C\big(\|\nabla\mathbf{u}\|^2_{L^2}+\|\nabla\mathbf{B}\|^2_{L^2}+\|P\|_{L^2}^2\big)
 +Ct\big(\|\nabla\mathbf{u}\|_{L^{2}}^2+\|\nabla\mathbf{B}\|_{L^{2}}^2+\|P\|_{L^{2}}^2\big)^2,
 \quad \text{for } t\in (\sigma(T),T),
\end{align}
where we have used Lemma \ref{l2.6} and the following estimates:
\begin{align*}
Ct\|\nabla\mathbf{u}\|_{L^3}^3+C\int t P|\nabla\mathbf{u}|^2\mathrm{d}\mathbf{x}
&\leq Ct(\|\nabla\mathbf{u}\|_{L^3}^3+\|P\|_{L^3}^3)
\notag\\
&\leq
Ct\big(\|\sqrt{\rho}\|_{L^3}^3\|\sqrt{\rho}\dot{\mathbf{u}}\|_{L^2}^3+\|P\|_{L^3}^3+\|\mathbf{B}\|_{L^6}^6\big)
\notag\\
&\leq
Ct\Big(A_2^\frac12(T)\|\sqrt{\rho}\|_{L^3}^3\|\sqrt{\rho}\dot{\mathbf{u}}\|_{L^2}^2+\|P\|_{L^3}^3+\|\nabla\mathbf{B}\|_{L^2}^4\Big)
\notag\\
&\leq
\frac18t\|\sqrt{\rho}\dot{\mathbf{u}}\|_{L^2}^2
 +\frac{\widehat{C}}{2}t\|P\|_{L^3}^3+Ct\|\nabla\mathbf{B}\|_{L^2}^4, \quad \text{for } t\in (\sigma(T),T).
\end{align*}
Noting that
\begin{align*}
 2\|P\|_{L^2}^2
 &=\int P\big(-2F-|\mathbf{B}|^2+2(2\mu+\lambda)\divv\mathbf{u}\big)\mathrm{d}\mathbf{x}
 \notag\\
 &\leq
 2\|P\|_{L^\frac{4\gamma}{4\gamma-1}}\|F\|_{L^{4\gamma}}
 +C\|P\|_{L^2}(\|\nabla\mathbf{u}\|_{L^2}+\|\mathbf{B}\|_{L^4}^2)
 \notag\\
 &\leq
 C\|\rho^\frac12\|_{L^2}\|\rho^{\gamma-\frac12}\|_{L^\frac{4\gamma}{2\gamma-1}}
 \|\nabla F\|_{L^{\frac{4\gamma}{2\gamma+1}}}
 +C\|P\|_{L^2}(\|\nabla\mathbf{u}\|_{L^2}+\|\nabla\mathbf{B}\|_{L^2})
 \notag\\
 &\leq
 C\|\rho^\frac12\|_{L^2}\|\rho^{\gamma-\frac12}\|_{L^\frac{4\gamma}{2\gamma-1}}
 \big(\|\sqrt{\rho}\|_{L^{4\gamma}}\|\sqrt{\rho}\dot{\mathbf{u}}\|_{L^2}
 +\|\mathbf{B}\|_{L^{4\gamma}}\|\nabla\mathbf{B}\|_{L^2})
 +C\|P\|_{L^2}(\|\nabla\mathbf{u}\|_{L^2}+\|\nabla\mathbf{B}\|_{L^2})
 \notag\\
 &\leq
  \|P\|_{L^2}^2
  +C\big(\|\sqrt{\rho}\dot{\mathbf{u}}\|_{L^2}^2+\|\nabla\mathbf{u}\|_{L^2}^2+\|\nabla\mathbf{B}\|_{L^2}^2
  +\|\nabla\mathbf{B}\|_{L^2}^{\frac{2\gamma-1}{\gamma}}\|\nabla\mathbf{B}\|_{L^2}^2\big),
\end{align*}
we have
\begin{equation}\label{ef4}
  \int_{\sigma(T)}^T\|P\|_{L^2}^2\mathrm{d}t
  \leq C\int_{\sigma(T)}^T\Big(\|\sqrt{\rho}\dot{\mathbf{u}}\|_{L^2}^2+\|\nabla\mathbf{u}\|_{L^2}^2
  +\|\nabla\mathbf{B}\|_{L^2}^2+A_1^\frac{2\gamma-1}{2\gamma}(T)\|\nabla\mathbf{B}\|_{L^2}^2\Big)\mathrm{d}t
  \leq C.
\end{equation}
An application of Gr\"onwall's inequality to \eqref{ef3} over $(\sigma(T),T)$,
combined with \eqref{ef4} and \eqref{3.21}, shows that
\begin{align}\label{ef5}
  \sup_{t\in[\sigma(T),T]}\big(t\|\nabla \mathbf{u}\|_{L^2}^2+t\|\nabla \mathbf{B}\|_{L^2}^2+t\|P\|_{L^2}^2\big)
  +\int_{\sigma(T)}^Tt\big(\|\sqrt{\rho} \dot{\mathbf{u}}\|_{L^2}^2
  +\|\Delta\mathbf{B}\|_{L^2}^2+\|P\|_{L^3}^3\big)\mathrm{d}t\leq C.
\end{align}
By arguing as in the derivation of \eqref{ef}, but with the weight $\sigma^3$ replaced by $t^\frac32$, and then integrating over $(\sigma(T),T)$, we get
\begin{align*}
&\sup_{t\in[\sigma(T),T]}\big(t^\frac32\|P\|_{L^3}^3\big)
  +\int_{\sigma(T)}^Tt^\frac32\|P\|_{L^4}^4\mathrm{d}t \\
  &\leq 
  C\int_{\sigma(T)}^Tt^\frac32\big(\|F\|_{L^4}^4+\||\mathbf{B}|^2\|_{L^4}^4\big)\mathrm{d}t
 +C\int_{\sigma(T)}^Tt^\frac12\|P\|_{L^3}^3\mathrm{d}t
  \notag\\
  &\leq
  C\sup_{t\in[\sigma(T),T]}\Big(t^\frac12\|\nabla\mathbf{u}\|_{L^{2}}+t^\frac12\|\nabla\mathbf{B}\|_{L^{2}}
+t^\frac12\|P\|_{L^{2}}\Big)
  \int_{\sigma(T)}^T t\Big(A_2^\frac12(T)\|\sqrt{\rho} \dot{\mathbf{u}}\|_{L^2}^2
  +A_1^\frac12(T)\|\Delta\mathbf{B}\|_{L^2}^2\Big)\mathrm{d}t
  \notag\\
  &\quad
  +C\sup_{t\in[\sigma(T),T]}\big(t^\frac32\|\nabla \mathbf{B}\|_{L^2}^4\big)\int_{\sigma(T)}^T\|\nabla \mathbf{B}\|_{L^2}^2\mathrm{d}t+C\notag\\
  &\leq C.
\end{align*}
Thus, by an argument similar to that leading to \eqref{3.37}--\eqref{3.38}, but with $\sigma^3$ replaced by $t^\frac32$, we obtain
\begin{align*}
&\sup_{t\in[\sigma(T),T]}\big(t^\frac32\|\sqrt{\rho}\dot{\mathbf{u}}\|_{L^2}^2
+t^\frac32\||\mathbf{B}||\nabla \mathbf{B}|\|_{L^2}^2\big)
 +\int_{\sigma(T)}^Tt^\frac32\big(\|\nabla \dot{\mathbf{u}}\|_{L^2}^2
   +\||\mathbf{B}||\Delta \mathbf{B}|\|_{L^2}^2\big)\mathrm{d}t
 \notag\\
 &\leq
 C\sup_{t\in[\sigma(T),T]}\Big(t^\frac12\|\nabla\mathbf{u}\|_{L^{2}}+t^\frac12\|\nabla\mathbf{B}\|_{L^{2}}
+t^\frac12\|P\|_{L^{2}}\Big)
  \int_{\sigma(T)}^T t\Big(A_2^\frac12(T)\|\sqrt{\rho} \dot{\mathbf{u}}\|_{L^2}^2
  +A_1^\frac12(T)\|\Delta\mathbf{B}\|_{L^2}^2\Big)\mathrm{d}t
\notag\\
  &\quad
   +C\sup_{t\in[\sigma(T),T]}\big(t^\frac12\|\nabla \mathbf{B}\|_{L^2}^2\big)\int_{\sigma(T)}^Tt\|\Delta \mathbf{B}\|_{L^2}^2\mathrm{d}t
   +C\int_{\sigma(T)}^Tt^\frac32\|P\|_{L^4}^4\mathrm{d}t+C
 \leq C.
\end{align*}
This along with Lemma \ref{2.2}, Lemma \ref{l2.6}, Young's inequality, \eqref{ooz}, and \eqref{3.61} implies that
\begin{align}\label{3.64}
&|b(t_2)-b(t_1)|\notag\\
&\leq C
 \int_{t_1}^{t_2}\big(\|F\|_{L^\infty}+\|\mathbf{B}\|_{L^\infty}^2\big)\mathrm{d}t
 \notag\\
  &\leq 
  C\int_{t_1}^{t_2}\|F\|_{L^{\frac{52}{3}}}^{\frac{13}{16}}\|\nabla F\|_{L^4}^{\frac{3}{16}}\mathrm{d}t
  +C\int_{t_1}^{t_2}\|\mathbf{B}\|_{L^4}\|\nabla \mathbf{B}\|_{L^2}^{\frac12}\|\Delta \mathbf{B}\|_{L^2}^{\frac12}\mathrm{d}t
  \notag\\
  &\leq 
   C\int_{t_1}^{t_2}\|\nabla F\|_{L^{\frac{52}{29}}}^{\frac{13}{16}}\|\nabla F\|_{L^4}^{\frac{3}{16}}\mathrm{d}t
   +C\int_{t_1}^{t_2}\|\nabla \mathbf{B}\|_{L^2}\|\Delta \mathbf{B}\|_{L^2}^{\frac12}\mathrm{d}t
   \notag\\
  &\leq 
  C\int_{t_1}^{t_2} \Big(\|\sqrt{\rho}\|_{L^{\frac{52}{3}}}^{\frac{13}{16}}
  \|\sqrt{\rho}\dot{\mathbf{u}}\|_{L^2}^{\frac{13}{16}}+\||\mathbf{B}||\nabla\mathbf{B}|\|_{L^{\frac{52}{29}}}^{\frac{13}{16}}\Big)
  (1+t)^\frac{3}{4}
  \Big(\|\sqrt{\rho}\dot{\mathbf{u}}\|_{L^2}^\frac{3}{16}+\|\nabla\dot{\mathbf{u}}\|_{L^2}^\frac{3}{16}
  +\||\mathbf{B}||\nabla\mathbf{B}|\|_{L^4}^{\frac{3}{16}}\Big)\mathrm{d}t
  \notag\\
  &\quad
   +C\int_{t_1}^{t_2}\|\nabla \mathbf{B}\|_{L^2}\|\Delta \mathbf{B}\|_{L^2}^{\frac12}\mathrm{d}t
    \notag\\
  &\leq
   C\Big(1+\sup_{t\in[\sigma(T),T]}\big(\sigma\|\nabla \mathbf{B}\|_{L^2}^2\big)
+\sup_{t\in[\sigma(T),T]}\big(\sigma^3\|\sqrt{\rho}\dot{\mathbf{u}}\|_{L^2}^2\big)\Big)
  \int_{\sigma(T)}^T\big(\sigma\|\sqrt{\rho}\dot{\mathbf{u}}\|_{L^{2}}^{2}
  +\sigma\|\Delta \mathbf{B}\|_{L^2}^2+\sigma^3\|\nabla\dot{\mathbf{u}}\|_{L^{2}}^{2}\big)\mathrm{d}t
   \notag\\
  &\quad
  +C\sup_{t\in[\sigma(T),T]}\big(\sigma\|\nabla \mathbf{B}\|_{L^2}^2\big)\int_{\sigma(T)}^T\sigma^3\||\mathbf{B}||\Delta\mathbf{B}|\|_{L^2}^2\mathrm{d}t
  +C\sup_{t\in[\sigma(T),T]}\big(t^\frac{3}{4}\|\sqrt{\rho}\dot{\mathbf{u}}\|_{L^2}\big)
 \|\rho\|_{L^\frac{26}{3}}^{\frac{13}{32}}
 \notag\\
  &\quad
 +C\Big(1+\sup_{t\in[\sigma(T),T]}\big(\sigma^3\||\mathbf{B}||\nabla \mathbf{B}|\|_{L^2}^2\big)\Big)\sup_{t\in[\sigma(T),T]}\big(t^\frac38\||\mathbf{B}||\nabla \mathbf{B}|\|_{L^2}\big)
 \bigg(\int_{\sigma(T)}^Tt^\frac32\||\mathbf{B}||\Delta \mathbf{B}|\|_{L^2}^2\mathrm{d}t\bigg)^\frac14
\bigg(\int_{t_1}^{t_2}1\mathrm{d}t\bigg)^\frac34
\notag\\
  &\quad
 +C\sup_{t\in[\sigma(T),T]}\Big(t^\frac{39}{64}\|\sqrt{\rho}\dot{\mathbf{u}}\|_{L^2}^\frac{13}{16}\Big)
\bigg(\int_{\sigma(T)}^Tt^\frac32\|\nabla\dot{\mathbf{u}}\|_{L^2}^2\mathrm{d}t\bigg)^\frac{3}{32}
\bigg(\int_{t_1}^{t_2}1\mathrm{d}t\bigg)^\frac{29}{32}\|\rho\|_{L^\frac{26}{3}}^{\frac{13}{32}}
 +\frac{a\tilde{c}(\theta-1)}{2(2\mu+\lambda)}(t_2-t_1)
\notag\\
&\leq \frac{a\tilde{c}(\theta-1)}{2\mu+\lambda}(t_2-t_1)+
C_5(\hat{\rho})C_0^\frac{19-\alpha}{32(7\alpha+5)}.
\end{align}
By choosing
\begin{equation*}
  N_0=C_5(\hat{\rho})C_0^\frac{19-\alpha}{32(7\alpha+5)}, \ \
  N_1=\frac{a\tilde{c}(\theta-1)}{2\mu+\lambda}, \ \
  \xi^*=1,
\end{equation*}
we have that
\begin{equation*}
  f(\xi)=-\frac{a\tilde{c}(\theta-1)}{2\mu+\lambda}\xi^\frac{\theta}{\theta-\gamma}\leq-N_1,
  \ \ \ \text{for all} \ \xi\geq \xi^*=1.
\end{equation*}
It thus follows from Lemma \ref{lzlo} that
\begin{equation}\label{3.65}
y(t)\leq \max \big\{\hat{\rho}^\theta, 1\big\} + N_{0}\leq\hat{\rho}^\theta+C_5(\hat{\rho})C_0^\frac{19-\alpha}{32(7\alpha+5)}
  \leq \Big(\frac{3}{2}\hat{\rho}\Big)^\theta
\end{equation}
provided
\begin{equation*}
  C_0\leq \varepsilon_{4,2}\triangleq
  \min\left\{\varepsilon_3,\bigg(\frac{\hat{\rho}^\theta}
  {C_5(\hat{\rho})}\bigg)^\frac{32(7\alpha+5)}{19-\alpha}\right\}.
\end{equation*}
Therefore, the desired \eqref{3.58} follows from \eqref{3.63} and \eqref{3.65}
as long as $C_0\leq\varepsilon_4\triangleq\min\{\varepsilon_{4,1},\varepsilon_{4,2}\}$.
\end{proof}

Now we are ready to prove Proposition $\ref{p3.1}$.

\begin{proof}[Proof of Proposition \ref{p3.1}.]
Proposition \ref{p3.1} follows from Lemmas \ref{l3.3}, \ref{l3.4}, and \ref{l3.6} if we select $\varepsilon=\varepsilon_4$.
\end{proof}

\section{Proof of Theorem 1.1}\label{sec4}

With the \textit{a priori} estimates established in Section \ref{sec3}, we are now in a position to prove Theorem \ref{t1.1}.

\begin{proof}[Proof of Theorem \ref{t1.1}.]
Let $(\rho_0, \mathbf{u}_0, \mathbf{B}_0)$ be initial data as described in the theorem.
For $\epsilon>0$, let $j_\epsilon=j_\epsilon(\mathbf{x})$ be the standard mollifier, and
define the approximate initial data $(\rho_0^\epsilon, \mathbf{u}_0^\epsilon, \mathbf{B}_0^\epsilon)$:
\begin{equation*}
\begin{cases}
\mathbf{u}_0^\epsilon=J_\epsilon*\mathbf{u}_0,~~
\mathbf{B}_0^\epsilon=J_\epsilon*\mathbf{B}_0,\\
\rho_0^\epsilon=J_\epsilon\ast\rho_0+\phi\epsilon,  \ \ \text{with} \
\phi=\phi(\epsilon, \mathbf{x})\triangleq \psi_{1/\epsilon}(\mathbf{x})+(1-\psi_{1/\epsilon}(\mathbf{x}))e^{-|\mathbf{x}|^2}\leq 1,
\end{cases}
\end{equation*}
where the cut-off function $\psi$ is given in \eqref{3.54}.
Then we have
\begin{align*}
\bar{x}^\alpha\rho_0^\epsilon\in L^1,\ \
(\rho_0^\epsilon-\phi\epsilon)\in H^2,\ \
\inf_{\mathbf{x}\in B_{1/\epsilon}}\{\rho_0^\epsilon(\mathbf{x})\}\geq \epsilon, \ \ (\mathbf{u}_0^\epsilon,\mathbf{B}_0^\epsilon)\in D^2\cap D^1,\ \
(\sqrt{\rho_0^\epsilon}\mathbf{u}_0^\epsilon,\mathbf{B}_0^\epsilon)\in L^2.
\end{align*}

By applying Lemma \ref{l2.1}, there exists a time $T_*>0$ such that the problem \eqref{a1}--\eqref{a3}
with initial data $(\rho_0^\epsilon, \mathbf{u}_0^\epsilon, \mathbf{B}_0^\epsilon)$
admits a unique strong solution $(\rho^\epsilon,{\bf u}^\epsilon,{\bf B}^\epsilon)$ on $\mathbb{R}^2\times(0,T_*]$ satisfying
\begin{equation*}
\bar{x}^\alpha\rho^\epsilon \in L^\infty(0,T; L^{1}), \ \
\rho^\epsilon-\phi\epsilon\in C([0,T_*];H^2),\
({\bf u}^\epsilon,{\bf B}^\epsilon)\in C([0,T_*];D^2\cap D^1), \ \inf\limits_{({\bf x},t)\in B_{1/\epsilon}\times [0,T_*]}\rho^\epsilon({\bf x},t)>0.
\end{equation*}
It follows from \eqref{1.10} and \eqref{3.2}--\eqref{3.4} that
\begin{equation*}
  A_1(0)=A_2(0)=0, \ \ A_3(0)\leq 2M^2\leq K,
\end{equation*}
which implies that there is $T_1\in(0,T_*]$ such that \eqref{3.5} holds for $T=T_1$.
Set
\begin{equation}\label{4.1}
  T^*=\sup\{T\,| \,\eqref{3.5} \ \text{holds}\}.
\end{equation}
Obviously $T^*\ge T_1>0$. We now claim
\begin{equation}\label{4.2}
  T^*=\infty.
\end{equation}
Otherwise, one gets from Lemma \ref{l3.5} that, for $0<T< T^*$,
\begin{align*}
  \sup_{t\in[0,T]}\|\bar{x}^\alpha\rho^\epsilon\|_{L^1}\leq C(T),
\end{align*}
which shows the tightness of $\{\rho^\epsilon\}$ at infinity: for any $\delta>0$, there exists $R>0$ such that
\begin{equation}\label{4.3}
  \sup_{\epsilon>0}\sup_{t\in[0,T]}\int_{\mathbb{R}^2\backslash B_R}\rho^\epsilon(t,\mathbf{x})\mathrm{d}\mathbf{x}\leq \delta.
\end{equation}
Note that Lemmas \ref{l3.3}, \ref{l3.4}, and \ref{l3.6} hold independently of the lower bound of initial density, the time of existence, and the parameter $\epsilon$.
Therefore, one infers from Proposition \ref{p3.1} that \eqref{3.6} is valid for all $0<T< T^*$ provided $C_0\leq \varepsilon$. The weak lower semicontinuity of the norms, along with
these uniform bounds, allows us to take the limit as $t\rightarrow T^*$.
Thus, we obtain a solution at time $T^*$ with regularity sufficient to serve as initial data. Applying Lemma \ref{l2.1} to this data yields an extension of solutions until some
$T^{**}>T^*$ such that \eqref{3.5} holds for any $0< T<T^{**}$, which contradicts \eqref{4.1}. Hence, \eqref{4.2} is true.

For any fixed $\tau$ and $T$ with $0<\tau<T<\infty$, it follows from Section \ref{sec3} that the approximate solutions $(\rho^\epsilon,\mathbf{u}^\epsilon,\mathbf{B}^\epsilon)$ satisfy the uniform bounds
\begin{align*}
\begin{cases}
  \rho^\epsilon \in L^\infty(0,T;L^1\cap L^\theta),\ \
  (\sqrt{\rho^\epsilon}\mathbf{u}^\epsilon,\mathbf{B}^\epsilon) \in L^\infty(0,T;L^2), \\
   (\nabla\mathbf{u}^\epsilon,\nabla\mathbf{B}^\epsilon)\in L^\infty(\tau,T;L^2)\cap L^2(0,T;L^2), \ \
    \sqrt{\rho^\epsilon}\dot{\mathbf{u}}^\epsilon \in L^\infty(\tau,T;L^2),
\end{cases}
\end{align*}
where $\theta>20\gamma$. A similar argument as that in \cite{LSX16} implies that
\begin{equation}\label{4.4}
  \lim_{\epsilon\rightarrow0}\|\nabla\mathbf{u}_0^\epsilon-\nabla\mathbf{u}_0\|_{L^2}
  + \lim_{\epsilon\rightarrow0}\|\sqrt{\rho_0^\epsilon}\mathbf{u}_0^\epsilon-\sqrt{\rho_0}\mathbf{u}_0\|_{L^2}
  + \lim_{\epsilon\rightarrow0}\|\mathbf{B}_0^\epsilon-\mathbf{B}_0\|_{H^1}=0.
\end{equation}
In addition, the equation $\eqref{a1}_1$ yields that
\begin{equation*}
  \partial_t\rho^\epsilon \ \ \text{is bounded in}  \ L^2(\tau,T;H^{-1}),
\end{equation*}
which, combined with \eqref{4.3} and the local Aubin--Lions compactness, shows that
(up to the extraction of a subsequence)
\begin{equation}\label{4.5}
  \rho^\epsilon-\phi\epsilon\rightarrow \rho \ \ \text{strongly in} \  C\big([0,T];L^q\big), \ \ \text{for any} \
  q\in[1,\theta).
\end{equation}
Furthermore, from the momentum equation $\eqref{a1}_2$ we have that
\begin{equation*}
  \partial_t(\rho^\epsilon\mathbf{u}^\epsilon) \ \ \text{is bounded in}  \ L^2\big(\tau,T;W^{-1,s}\big) \ \ \text{for some} \ s>1,
\end{equation*}
which, along with the uniform bound
\begin{equation*}
  \rho^\epsilon\mathbf{u}^\epsilon \in L^\infty\big(0,T;L^{\frac{2\theta}{\theta+1}}\big)
\end{equation*}
and the Aubin--Lions lemma, implies that (up to a subsequence)
\begin{equation}\label{4.6}
  \rho^\epsilon\mathbf{u}^\epsilon\rightarrow \mathbf{m} \ \ \text{strongly in} \
  C\big([\tau,T];L^{1+\zeta}_{\loc}\big) \ \ \text{for some small} \  \zeta.
\end{equation}
A similar discussion can be conducted for the magnetic field $\{\mathbf{B}^\epsilon\}$ by using \eqref{4.4}.
Define $\mathbf{u}=\mathbf{m}/\rho$ on $\{\rho>0\}$ and $\mathbf{u}=0$ on $\{\rho=0\}$.
By standard arguments (see, e.g., \cite{PL98,NS04,LSX16}), it follows from \eqref{4.5} and \eqref{4.6} that
\begin{equation*}
  \nabla\mathbf{u}^\epsilon\rightarrow \nabla\mathbf{u}, \ \
  \nabla\mathbf{B}^\epsilon\rightarrow \nabla\mathbf{B} \ \ \text{strongly in} \
  L^2(\tau,T;L^2(\{\rho>\delta\})) \ \ \text{for any} \ \delta>0.
\end{equation*}
Hence, passing to the limit $\epsilon\rightarrow0$ shows that the limit $(\rho, \mathbf{u},\mathbf{B})$ is indeed a weak solution in the sense of Definition \ref{d1.1} satisfying \eqref{1.12} on $\mathbb{R}^2\times(0,T]$ for any $0<T<T^*=\infty$.
\end{proof}

\section*{Conflict of interest}
The authors declare that they have no conflict of interest.

\section*{Data availability}
No data was used for the research described in the article.

\end{document}